\documentclass[a4paper]{amsart} 
\usepackage[dvipdfmx]{graphicx} 
\usepackage{times,txfonts,amsmath,amssymb} 
 
\usepackage{graphicx} 
\usepackage{pict2e} 
 
\theoremstyle{plain} 
\newtheorem{thm}{Theorem}[section] 
\newtheorem{lem}[thm]{Lemma} 
\newtheorem{cor}[thm]{Corollary} 
\newtheorem{prop}[thm]{Proposition} 
 
\theoremstyle{definition} 
\newtheorem{defn}[thm]{Definition} 
\newtheorem{rem}[thm]{Remark} 
 
\newtheorem{step}{Step}[] 
 
\newtheorem{algor}[thm]{Algorithm} 
 
\theoremstyle{remark} 
\newtheorem{oprob}[thm]{Open problem} 
\newtheorem{prob}[thm]{Problem} 
 
\newcommand{\cA}{{\mathcal A}} 
\newcommand{\cB}{{\mathcal B}} 
\newcommand{\cC}{{\mathcal C}} 
\newcommand{\cF}{{\mathcal F}} 
\newcommand{\cI}{{\mathcal I}} 
\newcommand{\cR}{{\mathcal R}} 
\newcommand{\cS}{{\mathcal S}} 
\newcommand{\fS}{{\mathfrak S}} 
 
\newcommand{\Vc}{{\mathcal V}} 
\newcommand{\Uc}{{\mathcal U}}

\newcommand{\bP}{{\mathbb P}} 
\newcommand{\bN}{{\mathbb N}} 
\newcommand{\bZ}{{\mathbb Z}} 
\newcommand{\bC}{{\mathbb C}} 
\newcommand{\bQ}{{\mathbb Q}} 
\newcommand{\bR}{{\mathbb R}} 
 
\DeclareMathOperator{\rk}{rk} 
\DeclareMathOperator{\Aut}{Aut} 
\DeclareMathOperator{\PGL}{PGL} 
\DeclareMathOperator{\Gal}{Gal} 
\DeclareMathOperator{\cone}{c} 
\DeclareMathOperator{\poin}{Poin} 
\DeclareMathOperator{\dH}{dH} 
\DeclareMathOperator{\pen}{pen} 
\DeclareMathOperator{\Van}{V} 
 
\newcommand{\lmultiset }{(} 
\newcommand{\rmultiset }{)} 
 
\newcommand{\algo}[6] 
{%\vspace{11pt}% \addtocounter{propo}{1} 
\begin{algor}{{\tt #1}{\rm (#2)}}\label{#1}\end{algor} 
\vspace{-6pt}\noindent{\it #3}. 
 
{\bf Input:} #4 
 
{\bf Output:} #5 
 
\newcounter{#1} 
\begin{list}{\textbf{\arabic{#1}.}}{\usecounter{#1}} 
#6\end{list}\vspace{3pt}} 
 
\title[Non-recursively freeness and non-rigidity] 
{Non-recursive freeness and non-rigidity of plane arrangements} 
\author[T.~Abe]{Takuro Abe} 
\address{ 
Takuro Abe 
\newline 
Department of Mechanical Engineering and Science, Kyoto University 
\newline 
Yoshida-Honmachi, Sakyo-ku, 
Kyoto 606-8501, JAPAN 
} 
\email{abe.takuro.4c@kyoto-u.ac.jp} 
\author[M.~Cuntz]{Michael Cuntz} 
\address{ 
Michael Cuntz 
\newline 
Institut f\"ur Algebra, Zahlentheorie und Diskrete Mathematik, 
\newline 
Fakult\"at f\"ur Mathematik und Physik, 
Leibniz Universit\"at Hannover 
\newline 
Welfengarten 1, D-30167 Hannover, GERMANY 
} 
\email{cuntz@math.uni-hannover.de} 
\author[H.~Kawanoue]{Hiraku Kawanoue} 
\address{Hiraku Kawanoue 
\newline 
Research Institute for Mathematical Sciences, 
Kyoto University 
\newline 
Kitashirakawa-Oiwakecho, Sakyo-ku, 
Kyoto 606-8502, JAPAN 
} 
\email{kawanoue@kurims.kyoto-u.ac.jp} 
\author[T.~Nozawa]{Takeshi Nozawa} 
\address{ 
Takeshi Nozawa 
\newline 
Maizuru National College of Technology 
\newline 
234 Shiraya, Maizuru 625-8511, JAPAN 
} 
\email{nozawa@kurims.kyoto-u.ac.jp} 
\thanks{ 
The first author's work was partially supported by Japan Society for the 
Promotion of Science 
Grant-in-Aid for Young Scientists (B), No.~24740012. 
\newline\indent 
The third author's work was partially supported by Japan Society for the 
Promotion of Science 
Grant-in-Aid for Young Scientists (B), No.~23740016. 
} 
\begin{document} 
\begin{abstract} 
In the category of free arrangements, inductively and recursively free 
arrangements are 
important. In particular, in the former, Terao's open problem asking whether 
freeness 
depends only on combinatorics is true. A long standing problem 
whether all free 
arrangements are recursively free or not 
was settled by the second author and Hoge 
very recently, by 
giving a free but non-recursively free plane arrangement consisting of 
27 planes. 
In this paper, we construct free but non-recursively free plane 
arrangements consisting of 13 and 15 planes, 
and show that the example with 13 planes is the smallest in the sense of 
the cardinality of planes. In other words, all free plane arrangements 
consisting of at most 12 planes are recursively free. To show this, 
we completely classify all free plane arrangements in terms of 
inductive freeness and three exceptions when the number of planes 
is at most 12. 
Several properties of the $15$ plane arrangement are proved 
by computer programs. Also, these two 
examples solve negatively a problem posed by Yoshinaga on 
the moduli spaces, (inductive) freeness and, rigidity 
of free arrangements. 
\end{abstract} 
\maketitle 
\begin{section}{Introduction} 
In the study of hyperplane arrangements, an important problem is 
to understand their freeness. 
In general, to determine whether 
a given arrangement is free or not is very difficult, 
and there is essentially only one way to check it, Saito's criterion (Theorem \ref{saito}). 
On the other hand, there is a nice way to construct a free arrangement 
from a given free arrangement, called 
the addition-deletion theorem (Theorem \ref{ad}). Since the empty arrangement 
is free, 
there is a natural question whether every free arrangement can be 
obtained, starting from the empty arrangement, 
by applying addition and deletion theorems. For simplicity, 
for the rest of this section, let us 
concentrate our interest on central arrangements in $\bC^3$. 
We say that a central arrangement $\cA$ is {\it inductively free} 
if it can be constructed by using only the addition theorem from 
the empty arrangement, and {\it recursively free} if we use 
both the addition and deletion theorems to construct it. 
A free arrangement which is not inductively free was found very soon (see 
Example 4.59 in \cite{OT} for example). 
However, a free but non-recursively free arrangement had not been 
found for a long time. 
Very recently, the second author and Hoge found a first such example \cite{CH}, 
which consists of 27 planes over $\bQ(\zeta)$ where $\zeta$ is a primitive fifth root in $\bC$. 
The aim of this paper is to give two new examples of free 
but non-recursively free arrangements in $\bC^3$ which have some further remarkable properties. 
Our examples consist of 13 and 15 planes defined over $\bQ$, and the former is the smallest. 
To show that there are no such arrangements when the number of planes is 
strictly less than 13, we also investigate the set of all 
free arrangements $\cA$ with $|\cA| \leq 12$. 
In other words, we give a complete classification of such 
free arrangements in terms of inductive, recursive freeness 
and three exceptions 
given in Definitions \ref{dH}, \ref{Pen} and \ref{443}. Now let us state 
our main theorem. 
\pagebreak[4] 
\begin{thm}\label{Main} 
Let $\cA$ be a central arrangement in $\bC^3$. 
\item[(1)] 
If $\cA$ is free with $|\cA|\leq12$, 
then $\cA$ is recursively free. More precisely, 
$\cA$ is either inductively free 
or one of the following arrangements 
characterized by their lattice structures. 
\begin{enumerate}\renewcommand{\labelenumi}{(\roman{enumi})} 
\item A dual Hesse arrangement, i.e., the arrangement $\cA$ 
with $|\cA|=9$, 
whose set $L_2(\cA)$ of codimension $2$ intersections 
consists of $12$ triple lines. 
\item A pentagonal arrangement, i.e., the arrangement $\cA$ 
with $|\cA|=11$ such that 
$L_2(\cA)$ consists of 
$10$ double lines, $5$ triple lines, 
$5$ quadruple lines and such that 
any $H\in\cA$ contains at most 
$5$ lines of $L_2(\cA)$. 
\item 
A monomial arrangement associated to the group $G(4,4,3)$, 
i.e., the arrangement $\cA$ 
with $|\cA|=12$, 
whose $L_2(\cA)$ consists of $16$ triple lines and $3$ 
quadruple lines. 
\end{enumerate} 
Moreover, the lattices of the arrangements in (i), (ii), and 
(iii) are realized over 
$\bQ[\sqrt{-3}]$, $\bQ[\sqrt{5}]$ and $\bQ[\sqrt{-1}]$, 
respectively. In particular, all free arrangements 
in $\bC^3$ are inductively free 
when $|\cA|\leq 8$ or $|\cA|=10$. 
\item[(2)] 
There exists a free but not recursively free 
arrangement $\cA$ 
over $\bQ$ with $|\cA|=13$. 
\end{thm} 
Our proof is based on combinatorial methods. 
We can check Theorem \ref{Main} by easy computations by hand, 
or just drawing a nice picture of our arrangement. 
Theorem \ref{Main} contains a characterization of 
free arrangements $\cA$ in $\bC^3$ 
with $|\cA|\leq 12$. 
We also find an example which is free but not recursively free with 15 planes by computer calculations. 
 
Previous research in this direction was performed especially from the viewpoint of 
the holy grail in the field of arrangements, Terao's open problem 
(see \cite{T}, \cite{T2}, \cite{p-hT-83}, \cite{Z}, \cite{p-gZ-89}, \cite{p-gZ-90}, \cite{OT}, \cite{p-Y-12}):

\begin{oprob}[Terao]\label{probter} 
Is the freeness of the logarithmic derivation module 
of any arrangement over a fixed field $K$ a purely combinatorial 
property of its intersection lattice? 
\end{oprob} 
 
To be more precise: are there arrangements $\cA_1,\cA_2$ 
in the same vector space $V$ such that $L(\cA_1)\cong L(\cA_2)$, 
$D(\cA_1)$ is free, and $D(\cA_2)$ is not free? 
We formulate Terao's open problem not as a conjecture here, because 
originally it was a problem posed by Terao, and there is some evidence 
on a non-dependency of the freeness on combinatorics; 
we quote Ziegler \cite{p-gZ-90}: ``We believe that Terao's conjecture 
over large fields is in fact false.'' 
 
Terao's open problem 
was checked when $|\cA| \leq 11$ in \cite{WY}, 
and $|\cA|\leq 12 $ in \cite{FV}. 
By Theorem \ref{Main}, we can give another proof 
of Terao's open problem 
when $|\cA|\leq 12$ which is originally due to Faenzi and Vall\`{e}s. 
\begin{cor}[Faenzi-Vall\`es \cite{FV}] 
Freeness of arrangements depends only on their combinatorics for central 
arrangements $\cA$ in $\bC^3$ with $|\cA|\leq12$. 
\end{cor} 
 
By investigating the structure of the classification in 
Theorem \ref{Main}, we can say that almost all the free 
arrangements with small exponents are inductively free. 
\begin{cor}\label{le4} 
Let $\cA$ be a free arrangement in $\bC^3$ 
with $\exp(\cA)=(1,a,b)$. 
If $\min(a,b)\leq4$, then 
$\cA$ is either inductively free or 
a dual Hesse arrangement appearing in Theorem \ref{Main}. 
\end{cor} 
 
Another formulation of Problem \ref{probter} is using the moduli space 
of all arrangements whose intersection lattice is a given lattice, see 
Definition \ref{modulispace} for details: 
\begin{oprob} 
Does a lattice $L$ exist such that the moduli space $\Vc(L)$ contains a free and a non-free arrangement? 
\end{oprob} 
 
Not much is known about this moduli space regarding freeness; 
Yuzvinsky \cite{p-sY-93} proved that the free arrangements form 
an open subset in $\Vc(L)$ (see also \cite[Theorem\ 1.50]{p-Y-12} 
for the case of dimension three). 
The first intuition one gets when working with these notions is, 
that either the moduli space is very big and the free arrangements 
are inductively free, or the moduli space is zero dimensional and 
all its elements are Galois conjugates (let us call such an 
arrangement \emph{rigid}). 
Yoshinaga proposed the following problem which is stronger 
than Terao's open problem 
since the property of being free is invariant 
under Galois automorphisms: 
 
\begin{prob}[Yoshinaga, {\cite[p.\ 20, (11)]{p-Y-12}}]\label{yopen} 
Is a free arrangement either inductively free, or rigid? In 
other words, does the 
following inclusion hold? 
\[ \{\text{Free arrangements}\} 
{\subset} \{\text{Inductively free}\} \cup \{\text{Rigid}\}.\] 
\end{prob} 
Beyond the fact that we find new small 
arrangements which are free but not recursively free, 
our examples with $13$ and $15$ planes 
provide a negative answer to Yoshinaga's 
problem \ref{yopen}, 
i.e., the following holds. 
 
\begin{thm} 
Let $L$ be the intersection lattice of our 13 or 15 plane arrangements. 
Then the moduli spaces $\Vc(L)$ of them 
are one-dimensional (hence not rigid), and not inductively free, but 
free. In particular, these give negative answers to 
Problem \ref{yopen}. 
\end{thm} 
 
We would like to emphasize that to our knowledge the two lattices presented here are the only two known such examples. Lattices satisfying these strong properties appear to be extremely rare. 
 
\begin{rem} 
The example with 13 planes was found by the second author in \cite{C} and 
by the fourth author in \cite{AKN} independently. 
The example with 15 planes was found and investigated by the second author in \cite{C}. 
The preprints \cite{AKN} and \cite{C} will not be published; 
all the results of \cite{AKN} and \cite{C} have been merged into this paper. 
\end{rem} 
 
The organization of our paper is as follows. 
In \S \ref{prel}, we introduce several definitions and results 
which will be used in this paper. 
In \S \ref{S9}, \S \ref{S11} and \S \ref{S12}, we prove Theorem \ref{Main} (1). 
In \S \ref{13}, we prove Theorem \ref{Main} (2). 
In \S \ref{15}, we exhibit the example with 15 planes having the same properties 
as the example with 13 planes by using algorithms presented 
in the same section. 
 
\subsection*{Acknowledgements} 
We are grateful to M.~Yoshinaga for his helpful comments. 
\end{section}

\begin{section}{Preliminaries}\label{prel} 
In this section, we summarize several definitions and results 
which will be used in this 
paper. 
For the basic reference on the arrangement theory, 
we refer to Orlik-Terao \cite{OT}. 
 
\begin{subsection}{Basic facts on arrangements} 
Let $V=\bC^n$. An {\it arrangement of hyperplanes} $\cA$ is 
a finite set of affine hyperplanes in $V$. An arrangement $\cA$ 
is {\it central} if every hyperplane is linear. 
For a hyperplane $H \subset V$, define 
$$ 
\cA \cap H=\{H \cap H' \neq \emptyset \mid H' \in \cA,\ H ' \neq H\}. 
$$ 
Hence $\cA \cap H$ is an arrangement in an $(n-1)$-dimensional vector space $H$. 
Let us define a {\it cone} $\cone\!\cA$ of an affine arrangement 
$\cA$ as follows. If $\cA$ is defined by a 
polynomial equation $Q=0$, then $\cone\!\cA$ is defined by 
$z\cdot{\cone}Q=0$, where ${\cone}Q$ is the homogenized 
polynomial of $Q$ by the new coordinate $z$. 
When $\cA$ is central, let us fix a defining linear form 
$\alpha_H \in V^*$ for each $H \in \cA$. From now on, 
let us concentrate our interest on 
central arrangements in $\bC^n$ when $n=2$ or $3$. 
So arrangements of lines or planes. 
Even when $n=3$, we view arrangements in $\bC^3$ as arrangements of lines in 
$P_{\bC}^2$ and call them line arrangements when there are 
no confusions. Let $S=S(V^*)=\bC[x_1,\ldots,x_n]$ be 
the coordinate ring of $V$. 
For a central arrangement $\cA$, define 
$$ 
D(\cA)=\left\{ \theta \in \bigoplus_{i=1}^n S \partial/\partial {x_i} \mid 
\theta(\alpha_H) \in S \alpha_H\ \text{for all}\ H \in \cA\right\}. 
$$ 
$D(\cA)$ is called the {\it logarithmic derivation module}. 
$D(\cA)$ is reflexive, but not free in general. 
We say that $\cA$ is free with {\it exponents} $(d_1,\ldots,d_n)$ if 
$D(\cA)$ has a homogeneous free basis $\theta_1,\ldots,\theta_n$ with 
$\deg \theta_i=d_i\ (i=1,\ldots,n)$. 
Here the {\it degree} of a homogeneous derivation 
$\theta=\sum_{i=1}^n f_i \partial/\partial {x_i}$ 
is defined by $\deg f_i$ for all non-zero $f_i$. 
Note that the Euler derivation $\theta_E= 
\sum_{i=1}^n x_i \partial/\partial x_i$ is contained in 
$D(\cA)$. 
In particular, it is easy to show that 
$ 
D(\cA)$ has $S \theta_E $ as its direct summand 
for a non-empty arrangement $\cA$. 
Hence $\exp(\cA)$ always contains $1$ if $\cA$ is not empty. 
To verify the freeness of $\cA$, Saito's criterion is essential. 
\begin{thm}[Saito's criterion, \cite{S}]\label{saito} 
Let $\theta_1,\ldots,\theta_n \in D(\cA)$. Then the following 
three conditions are equivalent. 
\item[(1)] 
$\cA$ is free with basis $\theta_1,\ldots,\theta_\ell$. 
\item[(2)] 
$\det [\theta_i(x_j)]=c \prod_{H \in \cA} \alpha_H$ 
for some non-zero $c \in \bC$. 
\item[(3)] 
$\theta_1,\ldots,\theta_n$ are all homogeneous derivations, 
$S$-independent and 
$\sum_{i=1}^n \deg\theta_i=|\cA|$. 
\end{thm} 
For a {\it multiplicity} $m\colon\cA \rightarrow \bZ_{>0}$, 
we can define the logarithmic 
derivation module $D(\cA,m)$ of the {\it multiarrangement} $(\cA,m)$ by 
$$ 
D(\cA,m)=\left\{ 
\theta \in \bigoplus_{i=1}^n S \partial/\partial {x_i} \mid 
\theta(\alpha_H) \in S \alpha_H^{m(H)}\ \text{for all}\ H \in \cA 
\right\}. 
$$ 
The freeness, exponents, and Saito's criterion 
for a multiarrangement can be defined in the same manner as 
in the case of arrangements without multiplicities. 
Note that the Euler derivation is not 
contained in $D(\cA,m)$ in general. 
Recall that every central (multi)\-arrangement in $\bC^2$ is free 
since $\dim_\bC \bC^2=2$ and $D(\cA ,m)$ is reflexive. Hence the first 
non-free central arrangement occurs when $n=3$. 
Define the {\it intersection lattice} $L(\cA)$ of $\cA$ by 
$$ 
L(\cA)=\left\{ \bigcap_{H \in \cB} H \mid \cB \subset \cA\right\}. 
$$ 
Reverse inclusion defines a poset structure on $L(\cA)$. 
The poset $L(\cA)$ is considered to be the combinatorial data of $\cA$. 
Define the {\it M\"{o}bius function} 
$\mu\colon L(\cA) \rightarrow \bZ$ by 
$$ 
\mu(V)=1,\quad 
\mu(X)=-\sum_{Y \in L(\cA),\ X \subsetneqq Y \subset V} \mu(Y) 
\quad (X \neq V). 
$$ 
Then the {\it characteristic polynomial} $\chi(\cA,t)$ 
of $\cA$ is defined by 
$\chi(\cA,t)=\sum_{X \in L(\cA)} \mu(X)\ t^{\dim X}$. 
It is known that 
$$ 
\chi(\cA,t)=t^n\poin(\bC^n \setminus 
\textstyle{\bigcup}_{H \in \cA} H,-t^{-1}). 
$$ 
Hence $\chi(\cA,t)$ is both a combinatorial and a topological 
invariant of an arrangement. 
The freeness of $\cA$ has implications for $\chi(\cA,t)$: 
\begin{thm}[Factorization, \cite{T2}]\label{factorization} 
Assume that a central arrangement $\cA$ is free with 
$\exp(\cA)=(d_1,\ldots,d_n)$. Then 
$$ 
\chi(\cA,t)=\prod_{i=1}^n (t-d_i). 
$$ 
\end{thm} 
Theorem \ref{factorization} implies that the algebraic structure of 
an arrangement might control its combinatorics and its topology. 
However, the converse is not true in general. For example, there 
are non-free arrangements in $\bC^3$ 
whose characteristic polynomials factorize over 
the ring of integers (see \cite{OT}). Hence it is natural to ask 
how much the algebraic structure and the combinatorics of arrangements 
are related. The most important problem 
in this context is Terao's open problem \ref{probter}, which is 
open even when $n=3$. In \cite{T}, 
Terao introduced a nice family of 
free arrangements in which the open problem \ref{probter} is true. 
To state it, let us introduce the following 
key theorem in this paper. 
\begin{thm}[Addition-Deletion, \cite{T}]\label{ad} 
Let $\cA$ be a free arrangement in $\bC^3$ with $\exp(\cA)=(1,d_1,d_2)$. 
\item[(1)] 
$($the addition theorem$)$. Let $H \not \in \cA$ be a linear plane. 
Then $\cB=\cA \cup\{H\}$ is free with $\exp(1,d_1,d_2+1)$ if and only if 
$|\cA \cap H|=1+d_1$. 
\item[(2)] 
$($the deletion theorem$)$. Let $H \in \cA$. 
Then $\cA'=\cA \setminus \{H\}$ is free with $\exp(1,d_1,d_2-1)$ if and only if 
$|\cA' \cap H|=1+d_1$. 
\end{thm} 
\begin{defn}\label{IR} 
\item[(1)] 
A central plane arrangement $\cA$ 
is {\it inductively free} 
if there is a filtration of subarrangements 
$\cA_1 \subset \cA_2 \subset \cdots \subset \cA_\ell=\cA$ such that 
$|\cA_i|=i\ (1\leq i\leq\ell)$ and every $\cA_i$ is free. 
\item[(2)] 
A central plane arrangement $\cA$ is {\it recursively free} 
if there is a sequence of arrangements 
$\emptyset=\cB_0,\cB_1,\cB_2,\ldots,\cB_t=\cA$ 
such that $|\cB_i\bigtriangleup\cB_{i+1}|=1\ (0\leq i\leq t-1)$ and 
every $\cB_i$ is free. 
\end{defn} 
Roughly speaking, an inductively free arrangement is 
a free arrangement constructed from an empty arrangement by 
using only the addition theorem, and 
a recursively free arrangement is a free arrangement constructed 
from an empty arrangement by using both 
the addition and deletion theorems. 
It is known that in the category of inductively free 
arrangements, the open problem \ref{probter} is true, but open in that of 
recursively free arrangements. 
\begin{rem} 
Theorem \ref{ad} and Definition \ref{IR} are different from 
those in an arbitrary dimensional case. They coincide when $n=3$ 
because every central arrangement in $\bC^2$ is free. 
For a general definition, see \cite{T} and \cite{OT} for example. 
\end{rem} 
\end{subsection} 
 
\begin{subsection}{Properties of free line arrangements} 
In this subsection, we 
introduce several special properties on free line arrangements which will be used 
from \S 3 to \S 6. 
\begin{defn} 
For $\ell\in\bZ_{\geq0}$, we define the sets $\cF_\ell$, 
$\cI_\ell$ and $\cR_\ell$ as follows. 
\begin{eqnarray*} 
\cF_{\ell}&=&\left\{ 
\text{free arrangement}\ \cA\ \text{in}\ \bC^3 
\ \text{with }\ |\cA|=\ell\right\}, 
\\ 
\cI_{\ell}&=&\left\{\cA\in\cF_\ell\mid 
\cA\colon\text{inductively free} 
\right\}, 
\\ 
\cR_{\ell}&=&\left\{\cA\in\cF_\ell\mid 
\cA\colon\text{recursively free} 
\right\}. 
\end{eqnarray*} 
Note that $\cI_{\ell}\subset\cR_{\ell}\subset\cF_{\ell}$ 
by definition. 
\end{defn} 
For the rest of this section, we will concentrate on 
central arrangements in $\bC^3$. 
One of the main purposes of this paper is to clarify the difference 
between $\cI_{\ell}$ and $\cF_{\ell}$ for $0\leq\ell\leq12$. 
Since 
$\cI_{\ell}=\cF_{\ell}$ for $\ell\leq1$, 
we may assume $\ell\geq2$. 
For $H\in\cA$, we denote 
$ 
n_{\cA,H}=\left|\cA\cap H\right| 
$. 
The following is 
the foundation stone of our analysis in this paper. 
\begin{thm}[\cite{A}]\label{ABT} 
Assume $\chi(\cA,t)=(t-1)(t-a)(t-b)$ for 
$a,b\in\bR$ 
and that there exists $H\in\cA$ such that $n_{\cA,H}>\min(a,b)$. 
Then, $\cA$ is free if and only if $n_{\cA,H}\in\{a+1,b+1\}$. 
\end{thm} 
\begin{defn} 
In view of Theorem \ref{ABT}, we define 
the subset $\cS_{\ell}$ of $\cF_{\ell}$ as 
$$\cS_{\ell}=\{\cA\in\cF_\ell 
\mid \exp(\cA)=(1,a,b),\ \max_{H\in\cA}n_{\cA,H}\leq\min(a,b)\}. 
$$ 
Note that $\cS_{\ell}\cap\cI_{\ell}=\emptyset$ by 
Theorems \ref{ad} and \ref{ABT}. 
\end{defn} 
\begin{lem}\label{Red} 
If $\cF_{\ell-1}=\cI_{\ell-1}$ 
(resp.~$\cR_{\ell-1}$), 
then 
$\cF_{\ell}=\cS_{\ell}\sqcup \cI_{\ell}$ 
(resp.~$\cS_{\ell}\cup \cR_{\ell}$). 
\end{lem} 
\begin{proof} 
Take $\cA\in\cF_{\ell}\setminus\cS_{\ell}$ 
and set $\exp(\cA)=(1,a,b)$. 
By Theorem \ref{ABT}, 
$\cA\in\cF_{\ell}\setminus\cS_{\ell}$ 
if and only if there exists 
a line $H\in\cA$ such that $n_{\cA,H}=a+1$ or $b+1$. 
Therefore, by the deletion theorem, 
we have $\cA\setminus\{H\}\in \cF_{\ell-1}$. 
Now the assertions above are clear. 
\end{proof} 
In the rest of this paper, 
we regard a central arrangement in $\bC^3$ 
as a line arrangement in $\bP_{\bC}^2$. 
For a line arrangement $\cA$ in $\bP_{\bC}^2$ 
and $P\in\bP_{\bC}^2$, we set 
$$ 
\cA_P=\{H\in\cA\mid P\in H\},\ 
\mu_P(\cA)=\left|\cA_P\right|-1,\ 
\mu_{\cA}=\sum_{P\in\bP_{\bC}^2}\mu_P(\cA). 
$$ 
Note that $\mu_P(\cA)$ is a reformulation of the 
M\"obius function for $L_2(\cA)$. 
If $\cA\neq\emptyset$, we can express $\chi({\cA},t)$ 
as follows by definition. 
$$\chi({\cA},t)=(t-1)\left\{ 
t^2-(|\cA|-1)(t+1) 
+\mu_{\cA}\right\}.$$ 
Concerning the set $\cS_{\ell}$, we have the following lemma. 
\begin{lem}\label{a-2} 
Let $\cA\in\cF_{\ell}$ with $\ell\geq2$ and 
$\exp(\cA)=(1,a,b)$. 
Assume $\mu_P(\cA)\geq\min(a,b)-1$ for some $P\in\bP_{\bC}^2$. 
Then we have $\cA\not\in\cS_{\ell}$.  In particular, 
$\cS_{\ell}=\emptyset$ for $2\leq\ell\leq6$. 
\end{lem} 
\begin{proof} 
We assume $a\leq b$ and $\cA\in\cS_{\ell}$. 
Let $P_0$ be a point in $\bP_{\bC}^2$. 
Suppose $\mu_{P_0}(\cA)\geq a$. 
If there exists $H\in \cA\setminus\cA_{P_0}$, we have 
$n_{\cA,H}\geq|\cA_{P_0}\cap H|\geq a+1$, which contradicts 
to $\cA\in\cS_{\ell}$. Therefore $\cA_{P_0}=\cA$, and it 
follows that $\exp(\cA)=(1,0,\ell-1)$. 
Thus $a=0$, but there exists $H\in\cA$ with $n_{\cA,H}\geq1$ 
since $\ell\geq2$, which contradicts to $\cA\in\cS_{\ell}$. 
Suppose $\mu_{P_0}(A)=a-1$. 
Since $|\cA_{P_0}|=a$ and $\cA\in\cS_{\ell}$, 
all intersection points of $\cA$ 
lie on $\bigcup_{H\in\cA_{P_0}}H$. 
It follows that $\cA$ is super solvable 
(see Definition 2.32 of \cite{OT} for the details) 
and $\exp(\cA)=(1,a-1,\ell-a)$, 
which contradicts to the condition on $a$. 
If $\ell\leq6$, the condition for the lemma automatically 
holds since 
$\mu_P(\cA)\geq1$ for some $P\in\bP_{\bC}^2$ and 
$a\leq\lfloor(\ell-1)/2\rfloor\leq2$. 
Therefore $\cS_{\ell}=\emptyset$. 
\end{proof} 
Now we introduce the invariant $F(\cA)$, which will be used 
to classify $\cS_\ell$. 
\begin{defn} 
Let $\cA$ be a line arrangement in $\bP_{\bC}^2$. We denote 
$$M_i(\cA)=\{P\in{\bP_{\bC}}^2\mid \mu_P(\cA)=i\}$$ 
and set the invariant $F(\cA)$ as 
$$ 
F(\cA)=[F_1(\cA),F_2(\cA),\ldots],\quad 
F_i(\cA)=\left|M_i(\cA)\right|\ (i=1,2,\ldots). 
$$ 
\end{defn} 
\begin{lem} 
The invariant $F(\cA)$ satisfies the following formulae. 
$$ 
\sum_{i} 
iF_i(\cA)=\mu_\cA 
,\ 
\sum_{i} 
(i+1)F_i(\cA)=\sum_{H\in\cA}n_{\cA,H} 
,\ 
\sum_{i} 
\binom{i+1}{2}F_i(\cA)=\binom{|\cA|}{2} 
. 
$$ 
\end{lem} 
\begin{proof} 
The left equation is clear by definition. 
Since $P\in M_i(\cA)$ is contained in $(i+1)$ lines of $\cA$, 
the middle equation holds. 
Finally, regarding all the intersection points of $\cA$ as 
the concentrations of the intersections of 2 lines of $\cA$, 
we have the right equation. 
\end{proof} 
Now we determine all the possibilities of $F(\cA)$ 
for $\cA\in\cS_\ell$ $(\ell\leq12)$. 
\begin{prop}\label{Classify} 
Let $\ell\in\bZ_{\leq12}$ and 
$\cA\in\cS_{\ell}$ with $\exp(\cA)=(1,a,b)$. 
Then we have 
$$(\ell,\min(a,b),F(\cA))\in\left\{ 
\begin{array}{ccc} 
(9, 4, [0, 12]), 
& 
(11, 5,[1, 14, 2]), 
& 
(11, 5,[4, 11, 3]), 
\\ 
(11, 5,[7, 8, 4]), 
& 
(11, 5,[10, 5, 5]), 
& 
(12, 5,[0, 16, 3]) 
\end{array} 
\right\}. 
$$ 
In particular, we have 
$\cS_{\ell}=\emptyset$, 
$\cF_{\ell}=\cI_{\ell}$ for $2\leq\ell\leq 8$ 
and $\cS_{10}=\emptyset$. 
\end{prop} 
\begin{proof} 
Note that $b=\ell-1-a$. 
We may assume $a\leq(\ell-1)/2$. 
By Lemma \ref{a-2}, we may assume 
$\ell\geq7$ and $F_i=0$ for $i\geq a-1$. 
Since $\chi({\cA},t)=(t-1)(t-a)(t-b)$ 
by Theorem \ref{factorization}, we have 
$\mu_{\cA}=ab+\ell-1=(\ell-1)(a+1)-a^2$. 
Also, since $\cA\in\cS_{\ell}$, we have 
$\sum_{H\in\cA}n_{\cA,H}\leq a\ell$. 
Thus we have the inequalities as follows 
$$ 
\sum_{i=1}^{a-2}iF_i(\cA)=(\ell-1)(a+1)-a^2, 
\quad 
\sum_{i=1}^{a-2}(i+1)F_i(\cA)\leq a\ell, 
\quad 
\sum_{i=1}^{a-2} 
\binom{i+1}{2}F_i(\cA)=\binom{\ell}{2}. 
$$ 
Solving above inequalities under the condition 
$0\leq a\leq(\ell-1)/2$ and $7\leq\ell\leq12$, we obtain only 
6 triplets $[\ell,a,F]$ appearing in the right hand side 
of the statement. Now $\cS_\ell=\emptyset$ for $2\leq\ell\leq8$ 
and $\cS_{10}=\emptyset$ 
are clear.  By Lemma \ref{Red}, we have 
$\cF_\ell=\cI_\ell$ for $2\leq \ell\leq8$. 
\end{proof} 
\begin{defn} 
For $H\in\cA$ and $i\in\bZ_{>0}$, we set 
$\mu_{\cA,H}=\sum_{P\in H}\mu_P(\cA)$ and 
$$ 
M_{i}(H,\cA)=M_i(\cA)\cap H,\ 
F_{H,i}(\cA)=\left|M_{i}(H,\cA)\right|,\ 
F_H(\cA)=[F_{H,1}(\cA),F_{H,2}(\cA),\ldots]. 
$$ 
\end{defn} 
\begin{lem} 
For $H\in\cA$, the invariant $F_{H}(\cA)$ satisfies 
the following formulae. 
$$ 
\sum_iF_{H,i}(\cA)=n_{\cA,H},\ 
\sum_iiF_{H,i}(\cA)=\mu_{\cA,H}=|\cA|-1,\ 
\sum_{H\in\cA}F_{H,i}(\cA)=(i+1)F_i(\cA). 
$$ 
\end{lem} 
\begin{proof} 
The formulae above are clear by definitions and the fact that 
$|\cA_P|=\mu_P(\cA)+1$. 
\end{proof} 
 
\medskip 
 
In sections from \S \ref{S9} to \S \ref{13}, 
we determine an arrangement $\cA\in\cS_{\ell}$ for $\ell\leq12$. 
Namely, we determine the lattice structures of $\cA$ 
up to the permutations $\fS_{\ell}$ of 
indices of hyperplanes and 
determine their realizations 
in $\bP_{\bC}^2$ up to the action of $\PGL(3,\bC)$. 
The hyperplanes of $\cA$ are denoted by $\cA=\{H_1,\ldots,H_\ell\}$, 
while the defining equation of each $H_i$ is denoted by $h_i$. 
The intersection points of $\cA$ 
satisfying $\cA_P=\{H_{a_i}\mid i\in I\}$ 
is denoted by $\{a_i\mid i\in I\}$. 
The line passing through $P$ and $Q$ is 
denoted by $\overline{PQ}$. 
For the coordinate calculation, 
we regard $\bP_{\bC}^2$ as 
the union of affine part ${\bC}^2$ and 
the infinity line $H_{\infty}$. 
 
\end{subsection} 
 
\begin{subsection}{Intersection lattices and moduli spaces} 
\begin{defn}[The moduli space of a graded lattice]\label{modulispace} 
Assume that $(L,\leq)$ is a graded lattice, i.e., all maximal chains in $L$ have the same length $r+1$ for some $r\in \bN$. This implies the existence of a map $\rk : L \rightarrow \{0,\ldots,r\}$ such that for fixed $k$ with $0\leq k\leq r$ the set $$L_k:=\{X\in L\mid \rk(X)=k\}$$ is an antichain. 
 
Now let $K$ be a field. 
Let $\Uc(L)$ be the set of arrangements $\cA$ with $|L_1|$ hyperplanes in $K^r$ such that there exists an isomorphism of posets 
$\varphi : L(\cA)\rightarrow L$ with $\varphi(\cA)=L_1$. 
It is well-known that $\Uc(L)$ is not an algebraic variety but still a locally closed set with respect to the Zariski topology. %(see for example \cite[]{}) 
For our purpose, the right object to study is the \emph{moduli space} $\Vc(L)$ of $L$ which we define as 
\[ \Vc(L) := \Uc(L) / \PGL_r(K), \] 
where the group $\PGL_r(K)$ acts on the linear forms of the arrangements. 
In other words, $\Vc(L)$ is the locally closed set of ``all arrangements with intersection lattice $L$ up to projectivities''. 
 
If $L$ is a poset, we write $\Aut(L)$ for the set of automorphisms of posets, i.e., the set of bijections preserving the relations in the poset. 
\end{defn}

\end{subsection} 
\end{section} 
\begin{section}{Determination of $\cS_9$}\label{S9} 
In this section, we show that $\cS_9$ consists of 
dual Hesse arrangements. 
\begin{subsection}{Lattice structure of $\cA\in\cS_9$}\label{S9L} 
We determine the lattice of $\cA\in\cS_9$. 
By Proposition \ref{Classify}, we have $F(\cA)=[0,12]$. 
Note that $F_{H}(\cA)=[0,4]$ for any $H\in\cA$ 
since $F(\cA)=[0,12]$ and $\sum_iiF_{H,i}(\cA)=\ell-1=8$. 
Concerning $M_2(H_9,\cA)$, we may set 
$$\{1,2,9\},\{3,4,9\},\{5,6,9\},\{7,8,9\} 
\in M_2(\cA).$$ 
Since $H_1\cap H_3$ lies on 
$H_5$, $H_6$, $H_7$ or $H_8$, 
we may set 
$\{1,3,5\}\in M_2(\cA)$ by symmetry. 
Since $H_1\cap H_7\neq H_1\cap H_8$, 
they are other 2 points of $M_2(H_1,\cA)$. 
Thus we may set 
$\{1,4,7\},\{1,6,8\}\in M_2(H_1,\cA)$ by symmetry 
of $(3,5)(4,6)$. Namely, 
$$\{1,3,5\},\{1,4,7\},\{1,6,8\}\in M_2(\cA).$$ 
Investigating 
$M_2(H_3,\cA)$, 
$M_2(H_4,\cA)$ and 
$M_2(H_2,\cA)$, 
it is easy to see 
$$ 
\{2,3,8\},\{3,6,7\}, 
\{2,4,6\},\{4,5,8\}, 
\{2,5,7\} 
\in M_2(\cA). 
$$ 
Now we obtain all the points of $M_2(\cA)$, thus 
the lattice structure of $\cA$ is determined. 
\end{subsection} 
\begin{subsection}{Realization of $\cA\in\cS_9$}\label{S9R} 
We determine the realization of $\cA\in\cS_9$ in $\bP_{\bC}^2$. 
We may set $H_9$ as the infinity line $H_\infty$, 
$h_1=x$, $h_2=x-1$, $h_3=y$, $h_4=y-1$ and 
$$ 
\{1,6,8\}=(0,p), 
\{2,5,7\}=(1,q), 
\{3,6,7\}=(r,0), 
\{4,5,8\}=(s,1) 
\quad 
(p,q,r,s\neq0,1). 
$$ 
Note that 
$(0,0),(1,q),(s,1)\in H_5$, 
$(1,1),(0,p),(r,0)\in H_6$, 
$(0,1),(1,q),(r,0)\in H_7$ and 
$(1,0),(0,p),(s,1)\in H_8$. 
Therefore we have 
$$sq=(1-r)(1-p)=r(1-q)=p(1-s)=1.$$ 
Solving these equations, we have 
$$(p,q,r,s)=(-\omega^2,-\omega,-\omega,-\omega^2),$$ 
where $\omega$ is a primitive third root of unity, 
and $h_i$ for $5\leq i\leq 8$ as follows. 
$$ 
h_5=y+\omega x,\ 
h_6=y+\omega x+\omega^2,\ 
h_7=y-\omega^2 x-1,\ 
h_8=y-\omega^2 x+\omega^2. 
$$ 
By this construction, for the permutation 
$\sigma\in\fS_{9}^\ast=\left\{\sigma\in\fS_{9} 
\mid \sigma(L(\cA))=L(\cA)\right\}$ preserving the 
lattice, there exists a $\PGL(3,\bC)$-action 
sending each $H_i$ to $H_{\sigma(i)}$, 
or sending each $H_i$ to $\overline{H_{\sigma(i)}}$, 
where $\overline{H_{i}}$ stands for the Galois conjugate 
of $H_i$ by $\Gal(\bQ[\sqrt{-3}]/\bQ)$. 
Note also that $\cA$ is transferred to 
$\overline{\cA}$ 
by 
$\left[(x,y,z)\mapsto(y,x,z)\right]\in\PGL(3,\bC)$, 
which sends $H_i$ to $\overline{H_{\mu(i)}}$ 
where 
$\mu=(1,3)(2,4)(7,8)\in\fS_{9}^\ast$. 
Thus $\cA$ is realized uniquely up to the $\PGL(3,\bC)$-action. 
\end{subsection} 
\begin{subsection}{Verifying $\cA\in\cS_9\subset\cR_9$}\label{S9V} 
We check the freeness of $\cA$ realized in \S \ref{S9R} and 
show that $\cA\in\cR_9$. 
We set $\cA_1=\cA\cup\{H_{10}\}$ where $h_{10}=x-y$. 
Then we have 
$$\cA_1\cap H_{10}=\{ 
(0,0),\ 
(1,1),\ 
((1-\omega^2)^{-1},(1-\omega^2)^{-1}),\ 
((1-\omega)^{-1},(1-\omega)^{-1}),\ 
H_{10}\cap H_{\infty}\}. 
$$ 
Since $\mu_{\cA_1}=\mu_{\cA}+5$, we have 
$\chi({\cA_1},t)=(t-1)(t-4)(t-5)$. 
By Theorem \ref{ABT}, we have 
$\cA_1\in\cF_{10}$ with $\exp(\cA_1)=(1,4,5)$, 
and hence $\cA\in\cS_9$. 
We set $\cA_2=\cA_1\setminus\{H_{\infty}\}$ and 
$\cA_3=\cA_2\setminus\{H_7\}$. 
Since $n_{\cA_1,H_{\infty}}=5$, we have 
$\cA_2\in\cF_9$ with $\exp(\cA_2)=(1,4,4)$. 
Since $n_{\cA_2,H_7}=5$, we have 
$\cA_3\in\cF_8=\cI_8$. Therefore 
$\cA_2\in\cI_9$, $\cA_1\in\cI_{10}$ and 
$\cA\in\cR_9$. 
 
\medskip 
 
In fact, to check whether $\cA\in\cF_9$ belongs to $\cS_9$ or not, 
we have only to check $F(\cA)$. 
\begin{lem} 
If $\cA\in\cF_9$ satisfies $F(\cA)=[0,12]$, 
then $\cA\in\cS_9$. 
\end{lem} 
\begin{proof} 
For any $H\in\cA$, 
since $9-1=\mu_{\cA,H}=2n_{\cA,H}$, we have 
$n_{\cA,H}=4$. 
\end{proof} 
\begin{defn}\label{dH} 
An arrangement in $\bC^3$ 
is called a {\it dual Hesse} arrangement 
if it is $\PGL(3,\bC)$-equivalent to 
$(\varphi_{\dH}=0)$, where 
$$ 
\varphi_{\dH}= 
(x^3-y^3)(y^3-z^3)(z^3-x^3). 
$$ 
\end{defn} 
It is easy to see that 
$\cA=(\varphi_{\dH}=0)$ 
satisfies $F(\cA)=[0,12]$. 
Therefore, 
$$\cS_9=\{\text{dual Hesse arrangements}\}\subset\cR_9.$$ 
\end{subsection} 
\begin{subsection}{Addition to $\cA\in\cS_9$} 
The structures of $\cF_9$ and $\cF_{10}$ are given as below. 
\begin{prop}\label{9-10} 
\item[(1)] 
Let $H\in \cA_1\in\cF_{10}$ such that 
$\cA=\cA_1\setminus\{H\}\in\cS_9$. 
Then, $\cA_1\in\cI_{10}$ and $\cA_1$ is unique up to 
the $\PGL(3,\bC)$-action. 
\item[(2)] 
$\cF_{9}=\cR_{9}=\cI_9\sqcup\cS_9$ and 
$\cF_{10}=\cI_{10}$. 
\end{prop} 
\begin{proof} 
\item[(1)] 
We may assume $\cA$ has the description 
as in \S \ref{S9L} and \S \ref{S9R}. 
By Theorem \ref{ad}, we have $n_{\cA_1,H}=5$ and hence 
$F_{H}(\cA_1)=[3,0,2]$. 
Since $H\not\in\cA$, $H$ is one of the following. 
\begin{eqnarray*} 
&& 
\overline{\{1,2,9\}\{3,6,7\}}, 
\overline{\{1,2,9\}\{4,5,8\}}, 
\overline{\{1,3,5\}\{2,4,6\}}, 
\overline{\{1,3,5\}\{7,8,9\}}, 
\\ && 
\overline{\{1,4,7\}\{2,3,8\}}, 
\overline{\{1,4,7\}\{5,6,9\}}, 
\overline{\{1,6,8\}\{2,5,7\}}, 
\overline{\{1,6,8\}\{3,4,9\}}, 
\\ && 
\overline{\{2,3,8\}\{5,6,9\}}, 
\overline{\{2,4,6\}\{7,8,9\}}, 
\overline{\{2,5,7\}\{3,4,9\}}, 
\overline{\{3,6,7\}\{4,5,8\}}. 
\end{eqnarray*} 
Recall that any $\sigma\in\fS_9^\ast$ 
is realized by the action of $\PGL(3,\bC)$ and 
$\Gal(\bQ[\sqrt{-3}]/\bQ)$. 
Therefore it suffices to show that 
$\fS_9^\ast$ acts transitively 
on the pairs in the above list. 
Observe that each points 
of $M_2(\cA)$ lies on two candidates of $H$. 
We denote 
the $\fS_9^\ast$-equivalence 
by the symbol ``$\sim$''. 
First note that 
$(\{1,2,9\},\{3,6,7\})\sim (\{3,6,7\},\{1,2,9\})$ 
by $(1,3)(2,6)(7,9)\in\fS_9^\ast$ and 
$(\{1,2,9\},\{3,6,7\})\sim (\{1,2,9\},\{4,5,8\})$ 
by $(3,4)(5,7)(6,8)\in\fS_9^\ast$. 
As the point transferred from $\{1,2,9\}$ has the 
same property as above, it follows that 
$(\{1,2,9\},\{3,6,7\})\sim (\{3,6,7\},\{4,5,8\})$. 
Namely, we have 
$$ 
(\{1,2,9\},\{3,6,7\})\sim 
(\{1,2,9\},\{4,5,8\})\sim 
(\{3,6,7\},\{4,5,8\}). 
$$ 
By applying 
$(2,3)(4,7)(5,9), (2,7)(4,9)(6,8), 
(2,6)(3,5)(8,9)\in\fS_9^\ast$, 
We have 
\begin{eqnarray*} 
(\{1,2,9\},\{3,6,7\}) 
\sim (\{1,3,5\},\{2,4,6\}) 
\sim (\{1,4,7\},\{2,3,8\}) 
\sim (\{1,6,8\},\{2,5,7\}). 
\end{eqnarray*} 
Therefore we have the following, which completes the proof 
of the uniqueness of $\cA_1$. 
\begin{eqnarray*} 
(\{1,2,9\},\{3,6,7\})&\sim& 
(\{1,3,5\},\{2,4,6\})\sim 
(\{1,3,5\},\{7,8,9\})\sim 
(\{2,4,6\},\{7,8,9\}) 
\\ 
&\sim& 
(\{1,4,7\},\{2,3,8\})\sim 
(\{1,4,7\},\{5,6,9\})\sim 
(\{2,3,8\},\{5,6,9\}) 
\\ 
&\sim& 
(\{1,6,8\},\{2,5,7\})\sim 
(\{1,6,8\},\{3,4,9\})\sim 
(\{2,5,7\},\{3,4,9\}). 
\end{eqnarray*} 
Note that $H_{10}$ in \S \ref{S9V} is $\overline{\{1,3,5\}\{2,4,6\}}$ 
and $\cA\cup\{H_{10}\}\in\cI_{10}$.  By the uniqueness of 
$\cA_1$, we conclude that $\cA_1\in\cI_{10}$. 
Therefore (1) is verified. 
\item[(2)] 
Since $\cF_8=\cI_8$ and $\cS_9\subset\cR_9$, we have 
$\cF_9=\cI_9\sqcup\cS_9=\cR_9$ by Lemma \ref{Red}. 
Let $\cA\in\cF_{10}$.  Since $\cS_{10}=\emptyset$, there exists 
$H\in\cA$ such that $\cA'=\cA\setminus\{H\}\in\cF_9=\cI_9\sqcup\cS_9$. 
If $\cA'\in\cI_9$, then $\cA\in\cI_{10}$. 
If $\cA'\in\cS_9$, we also have $\cA\in\cI_{10}$ by (1). 
Therefore we have $\cF_{10}=\cI_{10}$. 
\end{proof} 
We remark that now Theorem \ref{Main} is 
established for $|\cA|\leq10$ 
by Propositions \ref{Classify} and \ref{9-10}. 
We give the proof of Corollary \ref{le4}. 
\begin{proof}[Proof of Corollary \ref{le4}] 
The proof is by the induction on $\ell=|\cA|$.  If $\ell\leq10$, we 
have nothing to prove. Assume that $\ell\geq11$. 
If $\cA\in\cF_{\ell}\setminus\cS_{\ell}$, then 
$H\in\cA$ such that $\cA'=\cA\setminus\{H\}\in\cF_{\ell-1}$. 
Since $\exp(\cA')=(1,a-1,b)$ or $(1,a,b-1)$, 
we have $\cA'\in\cI_{\ell-1}$ by induction hypothesis, 
and hence $\cA\in\cI_{\ell}$. Thus we may assume 
$\cA\in\cS_{\ell}$. We set $a\leq b$ and take $H\in\cA$. 
By Lemma \ref{a-2}, we have $\mu_P(\cA)\leq a-2$ for any $P\in H$. 
By definition of $S_{\ell}$, we have $n_{\cA,H}\leq a$. 
However it is a contradiction since we have the following 
inequalities. 
$$11-1\leq\ell-1=\mu_{\cA,H}\leq (a-2)a 
\leq(4-2)\cdot4=8. 
\eqno\qed$$ 
\renewcommand{\qed}{} 
\end{proof} 
\end{subsection} 
\end{section} 
\begin{section}{Determination of $\cS_{11}$}\label{S11} 
In this section, we show that $\cS_{11}$ consists of pentagonal 
arrangements. 
\begin{subsection}{Absence of $\cA\in\cS_{11}$ with 
$F(\cA)=[1,14,2]$} 
Let $\cA\in\cS_{11}$. 
By Proposition \ref{Classify}, we have 
$F(\cA)=[1,14,2]$, 
$[4,11,3]$, $[7,8,4]$ or $[10,5,5]$. 
Suppose $F(\cA)=[1,14,2]$. Take $P\in M_{3}(\cA)$. 
Since $|M_1(\cA)\cup M_{3}(\cA)\setminus\{P\}|=2$ 
and $|\cA_P|=4$, there exists 
$H\in\cA_P$ such that 
$M_1(H,\cA)=\emptyset$ and $M_3(H,\cA)=\{P\}$. 
Then it follows that 
$11-1=\mu_{\cA,H}=0+2F_{H,2}(\cA)+3\cdot1$, 
which is impossible. 
Therefore we have $F(\cA)=[4,11,3], [7,8,4]$ or $[10,5,5]$. 
In the following subsections, we 
show that only the case $F(\cA)=[10,5,5]$ 
occurs, which corresponds to the case when 
$\cA$ is a pentagonal arrangement. 
\end{subsection} 
\begin{subsection}{Subarrangement $\cA'$ of $\cA$} 
In the case $F(\cA)=[4,11,3]$ or $[7,8,4]$, 
we construct a subarrangement 
$\cA'=\{H_1,\ldots,H_{10}\}$ of $\cA$ satisfying the following. 
$$ 
F(\cA')=[9,6,3],\quad 
n_{\cA',H_i}= 
\begin{cases} 
4 & i=1 \\ \leq5 & 2\leq i\leq 10 
\end{cases},\quad 
M_3(\cA')=\left\{\begin{array}{ccc} 
H_1\cap H_2, \\ H_1\cap H_3, \\ H_2\cap H_3 
\end{array}\right\}. 
\eqno{(\ast)} 
$$ 
Suppose $F(\cA)=[4,11,3]$. 
Note that $n_{\cA,H}=4,5$ for any $H\in\cA$, since 
$3\cdot3<10=\mu_{\cA,H}$. 
Thus $n_{\cA,H}=4,5$ for any $H\in\cA$. Since 
$\sum_{H\in\cA}n_{\cA,H}=2\cdot4+3\cdot11+4\cdot3 
=5\cdot11-2$, we may set $n_{\cA,H_1}=n_{\cA,H_2}=4$ and 
$n_{\cA,H_i}=5$ for $3\leq i\leq 11$. 
Note that $F_{H_i,3}(\cA)\geq2$ for $i=1,2$, since 
$2\cdot3+3\cdot1<10=\mu_{\cA,H}$. 
Thus we may set $M_3(\cA)=\{P_1,P_2,P_3\}$, 
$H_1=\overline{P_2P_3}$ and $H_2=\overline{P_1P_3}$. 
Since $n_{\cA,H_1}=4=\cA_{P_1}\cap H_1$, we have 
$\overline{P_1P_2}\in\cA$, which we set $H_3$. 
Since $|\bigcup_{P\in M_3(\cA)}\cA_{P}|=4\cdot3-3=9<11$, 
we may set $H_{11}\cap M_3(\cA)=\emptyset$. 
Then $F_{\cA,H_{11}}=[0,5,0]$. 
Now it is easy to check that 
$\cA'=\cA\setminus\{H_{11}\}$ 
satisfies the condition $(\ast)$. 
Suppose $F(\cA)=[7,8,4]$. 
Since $\sum_{H\in\cA}n_{\cA,H}=2\cdot7+3\cdot8+4\cdot4=5\cdot11-1$, 
we may assume $n_{\cA,H_1}=4$ and $n_{\cA,H_i}=5$ for $2\leq i\leq 11$. 
Note that $F_{H_1,3}(\cA)\neq1,4$ since 
$2\cdot3+3\cdot1<\mu_{\cA,H_1}=10<3\cdot4$. 
Thus $F_{H_1,3}(\cA)=2, 3$.  We set 
$M_3(\cA)=\{P_1, P_2, P_3, P_4\}$ so that 
$P_1\not\in H_1=\overline{P_2P_3}$. 
Since $|\cA_{P_1}\cap H_1|=4=n_{\cA,H_1}$, 
we have $P_2,P_3\in\cA_{P_1}$. Therefore 
we may set $H_2=\overline{P_1P_3}$ and 
$H_3=\overline{P_1P_2}$. 
Since $|\cA_{P_4}|=4$, we may set 
$M_3(\cA)\cap H_{11}=\{P_4\}$. 
Since $n_{\cA,H_{11}}=5$ and $F_{H_{11},3}(\cA)=1$, we have 
$F_{H_{11}}(\cA)=[1,3,1]$. 
Now it is easy to check that 
$\cA'=\cA\setminus\{H_{11}\}$ 
satisfies the condition $(\ast)$. 
\end{subsection} 
\begin{subsection}{Lattice structure of $\cA'$} 
First we determine $F_{H_i}(\cA')$ for $1\leq i\leq 10$. 
We have 
$n_{\cA',H}=4,5$ for any $H\in\cA'$ 
since $|M_3(H,\cA')|\leq2$ and 
$2+3\cdot2<9=\mu_{\cA',H}$. 
Since $n_{\cA',H_1}=4$ and $F_{3,H_1}(\cA')=2$, 
we have  $F_{H_1}(\cA')=[1,1,2]$. 
Since $\bigcup_{P\in M_3(\cA')}\cA_P'=4\cdot3-3=9$, 
we may set $M_{3}(H_{10},\cA')=\emptyset$, 
and hence $F_{H_{10}}=[1,4,0]$. 
Since $H_2\cap H_{10}\neq H_3\cap H_{10}$, 
we may set $H_2\cap H_{10}\in M_2(\cA')$. 
Since $F_{H_2,3}(\cA')=2$, we have 
$F_{H_2}(\cA')=[1,1,2]$. 
Since $\sum_{H\in\cA'}n_{\cA',H}=2\cdot9+3\cdot6+4\cdot3=5\cdot10-2$, 
we have $n_{\cA',H_i}=4$ for $i=1,2$ and 
$n_{\cA',H_i}=5$ for $3\leq i\leq 10$. 
We have $F_{H_3}(\cA')=[3,0,2]$ and 
$F_{H_i}(\cA')=[2,2,1]$ for $4\leq i\leq 9$ 
since $F_{H_3,3}(\cA')=2$ and $F_{H_i,3}(\cA')=1$. 
As a conclusion, we have the following. 
$$ 
F_{H_i}(\cA')= 
[1,1,2]\ (i=1,2),\ 
[3,0,2]\ (i=3),\ 
[2,2,1]\ (4\leq i\leq 9),\ 
[1,4,0]\ (i=10). 
$$ 
Now we determine the lattice structure of $\cA'$. We may set 
$$ 
M_3(\cA')=\{ 
P_1=\{2,3,4,5\}, 
P_2=\{1,3,6,7\}, 
P_3=\{1,2,8,9\} 
\}. 
$$ 
Note that 
$\{3,8\},\{3,9\},\{3,10\}\in M_1(\cA')$ 
since $F_{H_3}(\cA')=[3,0,2]$. 
By symmetry of $(4,5)$ or $(6,7)$, 
we may set 
$\{1,4,10\}, \{2,6,10\}\in M_2(\cA')$. 
Since $H_8\cap H_{10}$ lies on $H_5$ or $H_7$, 
we may set $\{5,8,10\}\in M_2(\cA')$ 
by symmetry of $(1,2)(4,6)(5,7)$. 
We also have $\{7,9,10\}\in M_2(\cA')$. 
Since 
$\{\{7,9,10\}\}\cup\left(\cA'_{P_1}\cap H_7\right)$ 
defines all intersection points on $H_7$, 
we have $H_7\cap H_8\in M_2(\cA)$. 
By the same reasoning for 
$\{\{5,8,10\}\}\cup\left(\cA'_{P_2}\cap H_8\right)$ 
on $H_8$, we have $H_4\cap H_8\in M_2(\cA)$. 
Thus we have $M_2(H_8,\cA')=\{\{5,8,10\},\{4,7,8\}\}$. 
Since 
$F_{H_i,2}(\cA')=2$ for $i=5,6,9$, 
The last point of $M_2(\cA)$ is $\{5,6,9\}$. 
Therefore $M_2(\cA')$, and hence $M_1(\cA')$, are as follows, 
which determine the lattice of $\cA'$. 
\begin{eqnarray*} 
M_2(\cA')&=&\left\{ 
\{1,4,10\},\ 
\{2,6,10\},\ 
\{4,7,8\},\ 
\{5,6,9\},\ 
\{5,8,10\},\ 
\{7,9,10\} 
\right\}, 
\\ 
M_1(\cA')&=&\left\{ 
\{1,5\},\ 
\{2,7\},\ 
\{3,8\},\ 
\{3,9\},\ 
\{3,10\},\ 
\{4,6\},\ 
\{4,9\},\ 
\{5,7\},\ 
\{6,8\} 
\right\}. 
\end{eqnarray*} 
\end{subsection} 
\begin{subsection}{Realization of $\cA'$} 
We determine the realization of $\cA'$ in $\bP_{\bC}^2$. 
We may set $H_{10}$ as the infinity line $H_{\infty}$, 
$P_1=(1,0)$, $P_2=(0,1)$ and $P_3=(0,0)$. 
Then 
$$ 
h_1=x,\ 
h_2=y,\ 
h_3=x+y-1,\ 
h_4=x-1,\ 
h_6=y-1. 
$$ 
Set $\{4,7,8\}=(1,p)$ and $\{5,6,9\}=(q,1)$ 
($p,q\neq0$). 
Then we have 
$$ 
h_5=x-(q-1)y-1 
,\ 
h_7=(p-1)x-y+1 
,\ 
h_8=px-y 
,\ 
h_9=x-qy. 
$$ 
Since 
$H_5\parallel H_8$ and 
$H_7\parallel H_9$, 
we have 
$p(q-1)=(p-1)q=1$. 
Therefore we conclude that 
$p=q=\zeta$ where $\zeta$ is a solution of 
$\zeta^2-\zeta-1=0$, and we may reset the equations as 
$$ 
h_5=\zeta x-y-\zeta 
,\ 
h_7=x-\zeta y+\zeta 
,\ 
h_8=\zeta x-y 
,\ 
h_9=x-\zeta y. 
$$ 
\end{subsection} 
\begin{subsection}{Absence of $\cA\in\cS_{11}$ 
with $F(\cA)=[4,11,3]$ or $[7,8,4]$} 
We show that we cannot extend the realization of $\cA'$ 
obtained above to $\cA$. 
Assume that $\cA=\cA'\cup\{H_{11}\}$ is realizable. 
Suppose $F(\cA)=[4,11,3]$. 
Recall that $F_{H_{11}}(\cA)=[0,5,0]$. 
Since $|M_1(\cA')\cap H_{11}|=5$ 
and 
$M_1(\cA')\subset 
\left\{\{1,5\},\{4,9\}\right\} 
\cup \bigcup_{i=3,6,7}H_i$, 
we have 
$H_{11}=\overline{\{1,5\}\{4,9\}} 
=\overline{(0,-\zeta)(1,\zeta^{-1})}$ and 
$h_{11}=(1-2\zeta)x+y+\zeta$. 
Therefore $H_{10}\cap H_{11}\in M_1(\cA)$, 
a contradiction. 
Suppose $F(\cA)=[7,8,4]$. 
Recall that $F_{H_{11}}(\cA)=[1,3,1]$. 
Thus $M_1(H_{11},\cA)=\{H_i\cap H_{11}\}$ for some $1\leq i\leq 10$. 
Since 
$n_{\cA',H_i}+1=n_{\cA,H_i}\leq5$, 
we have $i=1$ or $2$. We may set $H_{1}\cap H_{11}\in M_1(\cA)$ 
by the symmetry of the coordinates $x$ and $y$. Note that 
$|M_1(\cA')\cap H_{11}|=3$ and $|M_2(\cA')\cap H_{11}|=1$. 
In particular, $H_2\cap H_{11}=\{2,7\}$ or $\{2,6,10\}$. 
Assume that $H_2\cap H_{11}=\{2,7\}=(-\zeta,0)$. 
Then $M_2(\cA')\cap H_{11}=\{\{5,6,9\}\}$ or $\{\{5,8,10\}\}$. 
If $H_{11}\ni\{5,6,9\}=(\zeta,1)$, we have $h_{11}=x-2\zeta y+\zeta$. 
Therefore $H_{10}\cap H_{11}\in M_1(\cA)$, a contradiction. 
If $H_{11}\ni\{5,8,10\}$, 
we have $H_{11}\cap M_1(\cA')=\left\{ 
\{2,7\},\{3,9\},\{4,6\}\right\}$. 
Since $\{4,6\}=(1,1)$, 
we have $h_{11}=x-(\zeta+1)y+\zeta$, 
which contradicts to $H_{11}\parallel H_{8}$. 
Assume that $H_2\cap H_{11}=\{2,6,10\}$. Then 
we have 
$M_1(\cA')\cap H_{11}= 
\left\{\{3,8\},\{4,9\},\{5,7\}\right\}$. 
Since 
$\{4,9\}=(1,\zeta-1)$ 
and 
$\{5,7\}=(\zeta+1,\zeta+1)$ 
we have $H_{11}\not\parallel H_{2}$, 
a contradiction. 
 
\medskip 
 
Now we may assume that $F(\cA)=[10,5,5]$. 
\end{subsection} 
\begin{subsection}{Lattice structure of $\cA\in\cS_{11}$}\label{S11L} 
We determine the lattice of $\cA\in\cS_{11}$. 
First we show that 
$\overline{PQ}\in\cA$ for any $P,Q\in M_3(\cA)$, $P\neq Q$. 
Assume that there exist $P,Q\in M_3(\cA)$ such that 
$\overline{PQ}\not\in\cA$. Note that 
$\cA$ has $10+5+5=20$ intersection points, 
and $\cA_P\cup\cA_Q$ covers $4\cdot4+2=18$ of them. 
We set the left 2 intersection points in 
$\cA\setminus(\cA_P\cup\cA_Q)$ as $T_1$ and $T_2$. 
If $H=\overline{T_1T_2}\in\cA$, then 
$\{T_1,T_2\}\cap(\cA_P\cap H)\neq\emptyset$ 
since $|\cA_P\cap H|=4$ and $n_{\cA,H}\leq5$. 
However, it contradicts to the choice of $T_i$. 
If $\overline{T_1T_2}\not\in\cA$, then 
$(\cA_{P}\cup\cA_{Q})\cap(\cA_{T_1}\cup\cA_{T_2}) 
\neq\emptyset$ since $|\cA|=11$, 
$|\cA_{P}\cup\cA_{Q}|=8$ and 
$|\cA_{T_1}\cup\cA_{T_2}|\geq4$. 
It also contradicts to the choice of $T_i$. 
Next we determine $F_{H_i}(\cA)$ for $1\leq i\leq 10$. 
Note that $n_{\cA,H}=5$ for $H\in\cA$ since 
$\sum_{H\in\cA}n_{\cA,H}=2\cdot10+3\cdot5+4\cdot5=5\cdot11$. 
We also have $F_{H,3}(\cA)\leq2$ for $H\in\cA$ since 
$\mu_{\cA,H}=10<1\cdot2+3\cdot3$. 
It follows that, for $P,Q\in M_3(\cA)$ with 
$P\neq Q$, $\overline{PQ}\in\cA$ are distinct each other, forming 
$\binom{5}{2}=10$ lines of $\cA$. 
Thus we may assume $F_{H_i,3}(\cA)=2$, i.e., 
$F_{H_i}(\cA)=[2,1,2]$, for $1\leq i\leq 10$. 
Since $4\cdot5=\sum_{H\in\cA}F_{H,3}(\cA) 
=2\cdot10+F_{H_{11},3}(\cA)$, 
we have $F_{H_{11},3}(\cA)=0$, i.e., 
$F_{H_{11}}(\cA)=[0,5,0]$. 
Therefore, we have 
$$F_{H_i}(\cA)=[2,1,2]\ (1\leq i\leq 10), 
\quad F_{H_{11}}(\cA)=[0,5,0].$$ 
We investigate the lattice structure of $\cA$. 
We may set $M_3(\cA)=\{P_i\mid 1\leq i\leq5\}$ and 
\begin{eqnarray*} 
&& 
H_1=\overline{P_1P_2},\ 
H_2=\overline{P_1P_3},\ 
H_3=\overline{P_1P_4},\ 
H_4=\overline{P_1P_5},\ 
H_5=\overline{P_2P_3},\ 
\\ && 
H_6=\overline{P_2P_4},\ 
H_7=\overline{P_2P_5},\ 
H_8=\overline{P_3P_4},\ 
H_9=\overline{P_3P_5},\ 
H_{10}=\overline{P_4P_5}, 
\end{eqnarray*} 
or, in other words, $M_3(\cA)$ consists of 
the following five points. 
$$ 
P_1=\{1,2,3,4\},\ 
P_2=\{1,5,6,7\},\ 
P_3=\{2,5,8,9\},\ 
P_4=\{3,6,8,10\},\ 
P_5=\{4,7,9,10\}. 
$$ 
Since $H_1\cap H_{11}\in M_2(\cA)$ lies on $H_8$, $H_9$ or $H_{10}$, 
we may set $\{1,9,11\}\in M_2(\cA)$ by symmetry. 
Since $H_3\cap H_{11}\in M_2(\cA)$ lies on $H_5$ or $H_7$, 
we may set $\{3,5,11\}\in M_2(\cA)$ by symmetry of 
$(2,4)(5,7)(8,10)$. Investigating 
$H_{10}\cap H_{11}, H_{6}\cap H_{11}, H_{8}\cap H_{11} \in M_2(\cA)$ 
in this order, we have 
$\{2,10,11\},\{4,6,11\},\{7,8,11\}\in M_2(\cA)$. 
Thus $M_2(\cA)$ is determined. 
$$ 
M_2(\cA)=\left\{ 
\{1,9,11\},\ 
\{2,10,11\},\ 
\{3,5,11\},\ 
\{4,6,11\},\ 
\{7,8,11\} 
\right\}. 
$$ 
Now $M_2(\cA)$ and $M_3(\cA)$ are determined, which 
gives the lattice structure of $\cA$. 
\end{subsection} 
\begin{subsection}{Realization of $\cA\in\cS_{11}$}\label{S11R} 
We determine the realization of $\cA\in\cS_{11}$ 
in $\bP_{\bC}^2$. 
We may set $H_{11}$ as the infinity line $H_{\infty}$, 
$P_1=(0,1)$, $P_2=(0,0)$ and $P_3=(1,0)$. 
By definition of $H_1$, $H_2$, $H_5$ and 
the fact that $P_3\in H_9\parallel H_1$ and 
$P_1\in H_3\parallel H_5$ imply that 
$$h_1=x,\ h_2=x+y-1,\ h_3=y-1,\ h_5=y,\ h_9=x-1.$$ 
We set $P_4=(p,1)$ and $P_5=(1,q)$. 
Since $H_{10}=\overline{P_4P_5}\parallel H_2$ and 
$H_{4}=\overline{P_1P_5}\parallel 
H_{6}=\overline{P_2P_4}$, 
we have $p=q$ and $p(q-1)=1$. 
Thus we have $p=q=\zeta$ where $\zeta$ is a solution of 
$\zeta^2-\zeta-1=0$. 
The left defining equations $h_i$ of $H_i$ are as follows. 
$$ 
h_4=x-\zeta y+\zeta,\ 
h_6=x-\zeta y,\ 
h_7=\zeta x-y,\ 
h_8=\zeta x-y-\zeta,\ 
h_{10}=x+y-\zeta-1. 
$$ 
By this construction, for the permutation 
$\sigma\in\fS_{11}^{\ast} 
=\left\{\sigma\in\fS_{11} 
\mid \sigma(L(\cA))=L(\cA)\right\}$, 
there exists a $\PGL(3,\bC)$-action 
sending each $H_i$ to $H_{\sigma(i)}$, 
or sending each $H_i$ to $\overline{H_{\sigma(i)}}$, 
where $\overline{H_{i}}$ stands for the Galois conjugate 
of $H_i$ by $\Gal(\bQ[\sqrt{5}]/\bQ)$. 
Note also that $\cA$ is transferred to $\overline{\cA}$ 
by 
$\left[(x,y,z)\mapsto(\zeta x+y,x+\zeta y,z) 
\right]\in\PGL(3,\bC)$, 
which sends $H_i$ to $\overline{H_{\nu(i)}}$ 
where 
$\nu=(1,6,5,7)(2,10)(3,8,9,4)\in\fS_{11}^\ast$. 
Thus $\cA$ is realized uniquely up to the $\PGL(3,\bC)$-action. 
\end{subsection} 
\begin{subsection}{Verifying $\cA\in\cS_{11}\subset\cR_{11}$}\label{S11V} 
We check the freeness of $\cA$ realized in \S \ref{S11R} and 
show that $\cA\in\cR_{11}$. 
We set $\cA_1=\cA\cup\{H_{12}\}$ where 
$h_{12}=x-y$. Then we have 
$$\cA_1\cap H_{12}=\left\{ 
(0, 0),\ 
(1, 1),\ 
\left(\frac12, \frac12\right),\ 
(\zeta+1, \zeta+1),\ 
\left(\frac{\zeta+1}{2},\frac{\zeta+1}{2}\right),\ 
H_{12}\cap H_{\infty} 
\right\}.$$ 
Since $\mu_{\cA_1}=\mu_{\cA}+6$, we have 
$\chi({\cA_1},t)=(t-1)(t-5)(t-6)$. 
Thus $\cA_1\in\cF_{12}$ with $\exp(\cA_1)=(1,5,6)$ 
by Theorem \ref{ABT}, 
and hence $\cA\in\cS_{11}$. 
We set $\cA_2=\cA_1\setminus\{H_{\infty}\}$ and 
$\cA_3=\cA_2\setminus\{H_2\}$. 
Since $n_{\cA_1,H_{\infty}}=6$, we have 
$\cA_2\in\cF_{11}$ with $\exp(\cA_2)=(1,5,5)$. 
Since $n_{\cA_2,H_2}=6$, we have 
$\cA_3\in\cF_{10}=\cI_{10}$. 
Therefore, 
$\cA_2\in\cI_{11}$, $\cA_1\in\cI_{12}$ 
and 
$\cA\in\cR_{11}$. 
\begin{rem} 
Note that $\cA\in\cF_{11}$ satisfying $F(\cA)=[10,5,5]$ 
does not necessary belong to $\cS_{11}$.  For example, 
the arrangement $\cA$ defined by the equation 
$xyz(x^2-z^2)(y^2-z^2)$ 
$(x^2-y^2)(x-y+z)(x-y+2z)$ 
satisfies $F(\cA)=[10,5,5]$ 
but $\cA\in\cI_{11}$. 
\end{rem} 
\begin{defn}\label{Pen} 
The arrangement of Example 4.59 in \cite{OT} 
is the cone of the line arrangement consisted of 
5 sides and 5 diagonals of a regular pentagon, 
defined by the equation 
\begin{eqnarray*} 
\varphi_{\pen}&=& 
z\ (4x^2+2x-z) 
\\&& 
(x^4-10x^2y^2+5y^4+6x^3z-10xy^2z+11x^2z^2-5y^2z^2+6xz^3+z^4) 
\\&& 
(x^4-10x^2y^2+5y^4-4x^3z+20xy^2z+6x^2z^2-10y^2z^2-4xz^3+z^4). 
\end{eqnarray*} 
An arrangement in $\bC^3$ 
is called {\it pentagonal} 
if it is $\PGL(3,\bC)$-equivalent to 
$(\varphi_{\pen}=0)$. 
\end{defn} 
It is easy to see that 
$\cA=(\varphi_{\pen}=0)$ satisfies 
$\cA\in\cS_{11}$ and $F(\cA)=[10,5,5]$.  Therefore, 
$$\cS_{11}=\{\text{pentagonal arrangements}\}\subset\cR_{11}.$$ 
By the description in \S \ref{S11R}, the lattice of 
a pentagonal arrangement is realized over $\bQ[\sqrt{5}]$. 
\end{subsection} 
\begin{subsection}{Addition to $\cA\in\cS_{11}$} 
The structures of $\cF_{11}$ and $\cF_{12}$ are given as below. 
\begin{prop}\label{11-12} 
\item[(1)] 
Let $H\in \cA_1\in\cF_{12}$ such that 
$\cA=\cA_1\setminus\{H\}\in\cS_{11}$. 
Then, $\cA_1\in\cI_{12}$ and $\cA_1$ 
has two possibilities 
up to the $\PGL(3,\bC)$-action. 
\item[(2)] 
$\cF_{11}=\cR_{11}=\cI_{11}\sqcup\cS_{11}$ and 
$\cF_{12}=\cI_{12}\sqcup\cS_{12}$. 
\end{prop} 
\begin{proof} 
\item[(1)] 
We may assume that $\cA$ has the description 
as in \S \ref{S11L} and \S \ref{S11R}. 
By Theorem \ref{ad}, we have $n_{\cA_1,H}=6$. 
Note that 
$F_{H,3}(\cA_1)\leq1$ since $M_2(\cA)\subset H_{11}$ 
and $F_{H,4}(\cA_1)\leq1$ since 
$\mu_{\cA_1,H}=11<1\cdot4+4\cdot2$. 
Thus 
$F_{H}(\cA_1)=[1,5,0,0],[2,3,1,0],[3,2,0,1]$ or $[4,0,1,1]$. 
Suppose $F_{H}(\cA_1)=[1,5,0,0]$. 
By the description in \S \ref{S11L}, we have 
$$M_1(\cA)=\left\{ 
\{1,8\},\{4,8\},\{4,5\},\{5,10\},\{1,10\} 
\right\}\cup\left\{ 
\{2,6\},\{2,7\},\{3,7\},\{3,9\},\{6,9\} 
\right\}. 
$$ 
However, since one of the above two sets contains 
$3$ elements of $H\cap M_1(\cA)$, $H$ coincides with 
some $H_i\in\cA$, which is a contradiction. 
Therefore $F_{H}(\cA_1)\neq[1,5,0,0]$. 
Suppose $F_{H}(\cA_1)=[2,3,1,0]$. 
Note that the permutation 
$\rho=(1,5,8,10,4)(2,6,9,3,7)$ 
is an element of $\fS_{11}^\ast$, 
and the group $\langle\rho\rangle$ acts on 
$M_2(\cA)$ or $M_3(\cA)$ 
transitively. 
Since $\rho$ is realized by 
the action of $\PGL(3,\bC)$ and 
$\Gal(\bQ[\sqrt{5}]/\bQ)$, 
we may assume $\{1,9,11\}\in H$, i.e., $H\parallel(x=0)$. 
On the other hand, by the direct calculation, we have 
$$ 
M_1(\cA)= 
\left\{\begin{array}{lll} 
\{1,8\}=(0, -\zeta), & \{1,10\}=(0, \zeta+1), 
& \{2,6\}=(\zeta-1, 2-\zeta), 
\\ 
\{4,5\}=(-\zeta, 0), 
& \{5,10\}=(1+\zeta,0), 
& \{2,7\}=(2-\zeta,\zeta-1), 
\\ 
\{3,7\}=(\zeta-1, 1), 
& \{3,9\}=(1, 1), 
& \{4,8\}=(\zeta+1, \zeta+1). 
\\ 
\{6,9\}=(1, \zeta-1), 
\end{array}\right\}. 
$$ 
It is easy to see that no three points of $M_1(\cA)$ 
share the same $x$-coordinate, which contradicts to 
$F_{H,2}(\cA_1)=3$. 
Therefore $F_{H}(\cA_1)\neq[2,3,1,0]$. 
We show that, for each 
$\Gamma\in\{[3,2,0,1],[4,0,1,1]\}$ 
and for each $P\in M_3(\cA)$, exists 
the unique line $H\not\in\cA$ such that 
$P\in H$ and $F_H(\cA_1)=\Gamma$. 
First we assume $\{1,2,3,4\}\in H$. 
Suppose $\Gamma=[3,2,0,1]$. 
Since $H\not\in\cA$, 
we have $H\cap M_1(\cA)=\{\{5,10\},\{6,9\}\}$, 
and hence $H=(x+(1+\zeta)(y-1)=0)$. 
It is easy to check that $F_{H}(\cA_1)=[3,2,0,1]$. 
Suppose $\Gamma=[4,0,1,1]$. 
Since $H\not\in\cA$, 
we have $H\cap M_2(\cA)=\{\{7,8,11\}\}$, 
and hence $H=(\zeta x-y+1=0)$. 
It is easy to check that $F_{H}(\cA_1)=[4,0,1,1]$, 
We have seen that the unique $H$ exists for each $\Gamma$ 
if $P=\{1,2,3,4\}$.  To have the unique $H$ passing through 
another $P\in M_3(\cA)$, we have only to apply $\rho$ repeatedly. 
Therefore, $\cA_1$'s sharing the same $F_H(\cA_1)$ are transferred 
by the $\PGL(3,\bC)$-action. 
Next we show $\cA_1\in\cI_{12}$. 
Note that $\cA_1$ in \S \ref{S11V} 
satisfies $F_{H_{12}}(\cA_1)=[3,2,0,1]$ 
and $\cA_1\in\cI_{12}$. It follows that 
$\cA_1\in\cI_{12}$ when $F_{H}(\cA_1)=[3,2,0,1]$. 
We show that $\cA_1\in\cI_{12}$ when 
$F_{H}(\cA_1)=[4,0,1,1]$. We may assume 
$\cA_1=\cA\cup\{H\}$ with $H=(\zeta x-y+1=0)$. 
By Theorem \ref{ad}, we see that 
$\cA_1\in\cF_{12}$ with $\exp(\cA_1)=(1,5,6)$. 
We set $\cA_2=\cA_1\setminus\{H_{5}\}$ and 
$\cA_3=\cA_2\setminus\{H_9\}$. 
Since $n_{\cA_1,H_{5}}=6$, we have 
$\cA_2\in\cF_{11}$ with $\exp(\cA)=(1,5,5)$. 
Since $n_{\cA_2,H_9}=6$, we have 
$\cA_3\in\cF_{10}=\cI_{10}$. 
Therefore, we have 
$\cA_2\in\cI_{11}$ and $\cA_1\in\cI_{12}$. 
Thus we conclude that $\cA_1\in\cI_{12}$ 
for both cases of $F_{H}(\cA_1)$. 
\item[(2)] 
Since $\cF_{10}=\cI_{10}$ and $\cS_{11}\subset\cR_{11}$, we have 
$\cF_{11}=\cI_{11}\sqcup\cS_{11}=\cR_{11}$ by Lemma \ref{Red}. 
Let $\cA\in\cF_{12}\setminus\cS_{12}$.  Then, there exists 
$H\in\cA$ such that 
$\cA'=\cA\setminus\{H\}\in\cF_{11}=\cI_{11}\sqcup\cS_{11}$. 
If $\cA'\in\cI_{11}$, then $\cA\in\cI_{12}$. 
If $\cA'\in\cS_{11}$, we also have $\cA\in\cI_{12}$ by (1). 
Thus $\cF_{12}=\cI_{12}\sqcup\cS_{12}$. 
\end{proof} 
\end{subsection} 
\end{section} 
\begin{section}{Determination of $\cS_{12}$}\label{S12} 
In this section, we show that $\cS_{12}$ consists of 
monomial arrangements associated to the group $G(4,4,3)$. 
\begin{subsection}{Realization of $\cA\in\cS_{12}$} 
Let $\cA\in\cS_{12}$. 
By Proposition \ref{Classify}, $F(\cA)=[0,16,3]$. 
We show that $F_H(\cA)=[0,4,1]$ for any $H\in\cA$. 
Note that $n_{\cA,H}=5$ for any $H\in\cA$, 
since $\sum_{H\in\cA}n_{\cA,H}=3\cdot16+4\cdot3=5\cdot12$. 
We have $F_{H,3}(\cA)\leq1$ for $H\in\cA$, since 
$\mu_{\cA,H}=11<2\cdot3+3\cdot2$. 
If 
$M_3(H_0,\cA)=\emptyset$ for some $H_0\in\cA$, 
then 
$11=\mu_{P,H_0}=2F_{H_0,2}(\cA)$, 
a contradiction. Thus, for any $H\in\cA$, we have 
$F_{H,3}(\cA)=1$, and hence $F_H(\cA)=[0,4,1]$. 
We determine the realization of $\cA$ in $\bP_{\bC}^2$. 
We may set 
$$M_3(\cA)=\{\{1,2,3,4\},\{5,6,7,8\},\{9,10,11,12\}\}.$$ 
Since $F_{H_9,2}(\cA)=4$, we may assume 
$M_2(H_9,\cA)=\{\{1,5,9\},\{2,6,9\},\{3,7,9\},\{4,8,9\}\}$. 
We may set $\overline{\{1,2,3,4\}\{5,6,7,8\}} 
\not\in\cA$ as 
the infinity line $H_{\infty}$ and 
$$ 
h_1=x,\ h_2=x-1,\ h_3=x-p,\ h_4=x-q,\ 
h_5=y,\ h_6=y-1,\ h_7=y-s,\ h_8=y-t. 
$$ 
where 
$ 
|\{0,1,p,q\}|=|\{0,1,s,t\}|=4$. 
By choice of $H_9$, we have $s=p$, $t=q$ and $h_9=x-y$. 
We set 
$\alpha_i=\prod_{j=0}^3h_{4i-j}$ for $1\leq i\leq 3$. 
Since $F_{H,2}(\cA)=4$ for any $H\in\cA$, 
we have 
$\Van(\alpha_1,\alpha_2)=M_2(\cA)\subset\Van(\alpha_3)$. 
Therefore we have $\alpha_3\in\sqrt{(\alpha_1,\alpha_2)} 
=(\alpha_1,\alpha_2)$. 
Since $\deg(\alpha_i)=4$ for $i=1,2,3$, there exists 
$a,b\in\bC\setminus\{0\}$ such that $\alpha_3=a\alpha_1+b\alpha_2$. 
Since $\alpha_3\in(x-y)$, we have $b=-a$.  Therefore 
we may set $a=1, b=-1$. Now we obtain 
$$\alpha_3=\alpha_1-\alpha_2=x(x-1)(x-p)(x-q)-y(y-1)(y-p)(y-q).$$ 
Set $\beta=h_{10}h_{11}h_{12}=\alpha_3/(x-y)$.  Then we have 
$$ 
\beta=x^3+x^2y+xy^2+y^3-(p+q+1)(x^2+xy+y^2) 
+(pq+p+q)(x+y)-pq. 
$$ 
Since 
$\beta$ is a symmetric polynomial in $x$ and $y$, 
we may set 
$$ 
\beta=u^{-1}(x+uy+v)(ux+y+v)(x+y-2w). 
\quad(u,v,w,\in\bC,\ u\neq0). 
$$ 
Then we have $\{9,10,11,12\}=(w,w)$ and hence 
$v=-(1+u)w$. 
Comparing coefficients of $x^2y$, $x^2$, $x$ 
and constant terms, we obtain 
$1 = u+1+u^{-1}$, 
$-p-q-1 = -w(u+4+u^{-1})$, 
$pq+p+q = 3u^{-1}(1+u)^2w^2$ 
and $-pq = -2u^{-1}(1+u)^2w^3$. 
Therefore we have 
$$ 
u^2=-1,\quad 
p+q = 4w-1,\quad 
pq+p+q = 6w^2,\quad 
pq = 4w^3. 
$$ 
and hence $2w-1=0, \pm\sqrt{-1}$. 
Now it is easy to show that 
$$ 
\{1,p,q\}= 
\left\{1, 
\pm\sqrt{-1}, 1 \pm \sqrt{-1} 
\right\}, 
\left\{1, 
(1 + \sqrt{-1})/2, 
(1 - \sqrt{-1})/2 
\right\}. 
$$ 
These 3 possibilities of $(p,q)$ are identified 
by the actions 
$(x,y)\mapsto(x/p,y/p)$ or $(x,y)\mapsto(x/q,y/q)$, 
which corresponds to the changing of scale so as to set 
$H_2$, $H_3$ or $H_4$ to be $(x=1)$. 
Here we adopt $\{p,q\}=\{\sqrt{-1}, 1+\sqrt{-1}\}$. 
Then we have 
$$\alpha_3=(x-y)(x-\sqrt{-1}y- 1) 
(x+y-1-\sqrt{-1})(x+\sqrt{-1}y-\sqrt{-1}).$$ 
Now we have obtained the unique realization of $\cA\in\cS_{12}$ 
up to the $\PGL(3,\bC)$-action. 
\end{subsection} 
\begin{subsection}{Verifying $\cA\in\cS_{12}\subset\cR_{12}$} 
We set $\cA_1=\cA\cup\{H_{\infty}\}$. 
It is easy to see that 
$n_{\cA_1,H_{\infty}}=6$. 
Since $\mu_{\cA_1}=\mu_{\cA}+6$, we have 
$\chi({\cA_1},t)=(t-1)(t-5)(t-7)$. 
Thus $\cA_1\in\cF_{13}$ with $\exp(\cA_1)=(1,5,7)$ 
by Theorem \ref{ABT}, and hence $\cA\in\cS_{12}$. 
We set $\cA_2=\cA_1\setminus\{H_{9}\}$ and 
$\cA_3=\cA_2\setminus\{H_{10}\}$. 
Since $n_{\cA_1,H_{9}}=6$, we have 
$\cA_2\in\cF_{12}$ with $\exp(\cA_2)=(1,5,6)$. 
Since $n_{\cA_2,H_{10}}=6$, we have 
$\cA_3\in\cF_{11}=\cR_{11}$. 
Therefore, $\cA_2\in\cR_{12}$, $\cA_1\in\cR_{13}$ 
and $\cA\in\cR_{12}$. 
 
\medskip 
 
In fact, to check whether $\cA\in\cF_{12}$ belongs to $\cS_{12}$ or not, 
we have only to check $F(\cA)$. 
\begin{lem} 
If $\cA\in\cF_{12}$ satisfies $F(\cA)=[0,16,3]$, 
then $\cA\in\cS_{12}$. 
\end{lem} 
\begin{proof} 
If $\cA\not\in\cS_{12}$, there exists $H_1\in\cA$ 
such that $n_{\cA,H_1}\geq6$.  Since 
$\sum_{H\in\cA}n_{\cA,H}=5\cdot12$, 
there exists $H_2\in\cA$ 
such that $n_{\cA,H_2}\leq4$.  Since $F(\cA)=[0,16,3]$, 
we have $F_{H_2}(\cA)=[0,1,3]$. Take $H_3\in\cA\setminus\{H_2\}$ 
such that $\mu_{H_2\cap H_3}(\cA)=2$.  Then, since 
$H_3\cap M_3(\cA)=\emptyset$, we have 
$12-1=\mu_{\cA,H_3}=2n_{\cA,H_3}$, a contradiction. 
\end{proof} 
\begin{defn}\label{443} 
An arrangement in $\bC^3$ is called 
a monomial arrangement associated to 
the group $G(4,4,3)$ (see B.1 of \cite{OT}), 
if it is $\PGL(3,\bC)$-equivalent to 
$(\varphi_{4,4,3}=0)$, where 
$$\varphi_{4,4,3}= 
(x^4-y^4)(y^4-z^4)(z^4-x^4). 
$$ 
\end{defn} 
It is easy to see that 
$\cA=(\varphi_{4,4,3}=0)$ 
satisfies $F(\cA)=[0,16,3]$. 
Therefore, 
$$\cS_{12}=\{ 
\text{monomial arrangements 
associated to the group\ }G(4,4,3)\}\subset\cR_{12}.$$ 
By Proposition \ref{11-12}, we have 
$\cF_{12}=\cS_{12}\sqcup\cI_{12}=\cR_{12}$. 
Thus (1) of Theorem \ref{Main} is verified. 
\end{subsection} 
\end{section} 
\begin{section}{An arrangement with $13$ hyperplanes}\label{13} 
In this section, we construct a free arrangement of 13 planes 
in $\bC^3$ which is neither recursively free or rigid. 
\begin{subsection}{Definition of $\cA$}\label{13def} 
Let $\cA_\lambda'=\{H_1,\ldots,H_{12}\}$ be the line arrangement 
in $\bC^2$, where $H_i$ is defined by $h_i$ below 
for each $1\leq i\leq 12$, with a parameter $\lambda\in\bC$. 
 
\noindent 
\begin{minipage}{0.5\hsize} 
\begin{eqnarray*} 
h_1&=&-\sqrt{3}x-y+{\lambda}+1, 
\\ 
h_2&=&2y+{\lambda}+1, 
\\ 
h_3&=&\sqrt{3}x-y+{\lambda}+1, 
\\ 
h_4&=&\sqrt{3}x-y+{\lambda}-2, 
\\ 
h_5&=&-\sqrt{3}x-y+{\lambda}-2, 
\\ 
h_6&=&2y+{\lambda}-2, 
\\ 
h_7&=&2y-2{\lambda}+1, 
\\ 
h_8&=&\sqrt{3}x-y-2{\lambda}+1, 
\\ 
h_9&=&-\sqrt{3}x-y-2{\lambda}+1, 
\\ 
h_{10}&=&({\lambda}+1)y+\sqrt{3}(1-{\lambda})x-{\lambda}^2+{\lambda}-1, 
\\ 
h_{11}&=&\sqrt{3}{\lambda}x+({\lambda}-2)y-{\lambda}^2+{\lambda}-1, 
\\ 
h_{12}&=&(1-2{\lambda})y-\sqrt{3}x-{\lambda}^2+{\lambda}-1. 
\end{eqnarray*} 
\end{minipage} 
\begin{minipage}{0.5\hsize} 
\begin{center} 
\begin{picture}(180,180)(0,0) 
\setlength\unitlength{0.5pt} 
 \put(10,210){\line(1,0){340}} 
 \put(-20,204){$H_6$} 
 \put(10,187){\line(1,0){340}} 
 \put(-20,179){$H_7$} 
 \put(10,147){\line(1,0){340}} 
 \put(-20,140){$H_2$} 
 \put(65,315){\line(3,-5){170}} 
 \put(230,10){$H_5$} 
 \put(93,315){\line(3,-5){160}} 
 \put(260,40){$H_9$} 
 \put(129,335){\line(3,-5){150}} 
 \put(285,75){$H_1$} 
 \put(124,30){\line(3,5){150}} 
 \put(274,290){$H_4$} 
 \put(96,30){\line(3,5){140}} 
 \put(230,275){$H_8$} 
 \put(66,60){\line(3,5){150}} 
 \put(212,320){$H_3$} 
\qbezier(62,67)(192,167)(322,267) 
 \put(325,275){$H_{11}$} 
\qbezier(94,233)(218,187)(342,141) 
 \put(59,235){$H_{10}$} 
\qbezier(136,336)(156,210)(176,84) 
 \put(165,60){$H_{12}$} 
 \end{picture} 
{{\sc Figure 1.}\refstepcounter{figure} 
\quad$\cA_\lambda'$ with $\lambda=2/3$.} 
\end{center} 
\end{minipage} 
 
\medskip 
 
Note that $\cA_\lambda'$ is symmetric with respect to 
the rotation with angle $-2\pi/3$, which 
defines the action on $\cA_\lambda'$ given by 
$H_{3i}\mapsto H_{3i-2}\mapsto H_{3i-1}\mapsto H_{3i}$. 
Let $\cA_\lambda=\cone\!\cA_\lambda'$ be the cone of $\cA_\lambda'$. 
We denote the infinity line $H_{\infty}$ by $H_{13}$. 
$\cA_\lambda$ is denoted by $\cA$ if there is no confusion. 
 
For generic $\lambda\in\bC$, we see by calculation 
(or by reading off from the figure) that 
\begin{eqnarray*} 
M_2(\cA)&=&\left\{\{1,6,8\},\{2,4,9\},\{3,5,7\}\right\}, 
\\ 
M_3(\cA)&=&\{ 
\{1,5,9,13\}, 
\{2,6,7,13\}, 
\{3,4,8,13\} 
\}\subset H_{13}, 
\\ 
M_4(\cA)&=&\left\{ 
\{1,4,7,10,11\}, 
\{2,5,8,11,12\}, 
\{3,6,9,10,12\} 
\right\}. 
\end{eqnarray*} 
Other intersection points of $\cA$ form $M_1(\cA)$. 
We denote this lattice structure by $L_0$. 
Note that, if $\lambda\in\bQ$, 
$\cA_\lambda$ is defined over $\bQ$ after 
the $\PGL(3,\bC)$-action $[x\mapsto x/\sqrt{3}, y\mapsto y]$. 
Therefore $L_0$ is realized over $\bQ$. 
This section is devoted to prove the following proposition. 
\begin{prop}\label{13Descrip} 
\item[(1)] 
If $\lambda\in\{0,1,-\omega^{\pm1}\}$, then $|\cA_\lambda|<13$. 
\item[(2)] 
If $\lambda\in\{-1,2^{\pm1}\}$, 
then $\cA_\lambda\in\cI_{13}$ with $\exp(\cA_{\lambda})=(1,5,7)$. 
\item[(3)] 
Set $Z=\{0,\pm1,-\omega^{\pm1},2^{\pm1}\}$. 
If $\lambda\in\bC\setminus Z$, then $\cA_\lambda\in\cS_{13}$ 
with $\exp(\cA_{\lambda})=(1,6,6)$. 
\item[(4)] 
The moduli space $\Vc_\bC(L_0)$ of the lattice $L_0$ is 
a one dimensional set $\Vc_\bC(L_0)= \left(\bC\setminus Z\right)/\sim$ 
with each $\lambda\in \Vc_\bC(L_0)$ parametrizing an arrangement 
$\cA_\lambda$, where 
$\sim$ stands for the equivalence relation generated by 
the relations $\lambda\sim\lambda^{-1}$ and $\lambda\sim1-\lambda$. 
\item[(5)] 
The symmetry group of $L_0$ is 
$\Aut(L_0)\cong\bZ/3\bZ\rtimes\fS_3$. 
\item[(6)] 
Set $W=\left\{\pm\sqrt{-1}, 1\pm\sqrt{-1}, (1\pm\sqrt{-1})/2, 
(\pm1\pm\sqrt{5})/2, (3\pm\sqrt{5})/2 \right\}\subset\bC\setminus Z$. 
Then, $\cA_\lambda\in\cS_{13}\setminus \cR_{13}$ if and only if 
$\lambda\in \bC\setminus\left(Z\sqcup W\right)$. 
\end{prop} 
\end{subsection} 
\begin{subsection}{Degeneration of $\cA$}\label{degen} 
We show (1) and (2) of Proposition \ref{13Descrip}, and 
that $L(\cA_\lambda)=L_0$ for $\lambda\in\bC\setminus Z$. 
 
By the symmetry of the rotation, it is easy to see that 
$|\{H_i\mid 1\leq i\leq 9\}|=9\Leftrightarrow|\{H_2,H_6,H_7\}|=3 
\Leftrightarrow \lambda\neq0,1$. We also see that 
$H_{10}, H_{11}, H_{12}$ are non-parallel to $x$-axis if 
$\lambda\neq0,1$, distinct each other 
if $\lambda\neq-\omega^{\pm1}$, and all coincide 
if $\lambda=-\omega^{\pm1}$, 
where $\omega$ is the primitive third root of unity. 
Therefore it follows that 
$|\cA|=13\Leftrightarrow\lambda\in\bC\setminus\{0,1,-\omega^{\pm1}\}$. 
 
Assume that $\lambda\in\bC\setminus\{0,1,-\omega^{\pm1}\}$. 
The lattice structure is corrupted only when 
the origin lies on $H_i$ for $i=2,6,7,10$. 
The cases $i=2,6,7$ correspond to 
$\lambda=-1,2,1/2$ respectively. 
The case $i=10$ corresponds to 
$\lambda=-\omega^{\pm1}$. 
Therefore we have 
$L(\cA)=L_0$ for $\lambda\in\bC\setminus Z$. 
 
When $\lambda\in\{-1,2^{\pm1}\}$, 
the origin lies on $H_i$ for some $i\in\{2,6,7\}$. 
As the origin gives rise to the new intersection point 
$\{3\lceil i/3\rceil-j\mid0\leq j\leq2\}\in M_2(\cA)$, 
we have $F(\cA)=[18,4,3,3]$. Since $\mu_{\cA}=47$, we have 
$\chi(\cA,t)=(t-1)(t-5)(t-7)$. 
We set $\{i,j_1,j_2\}=\{2,6,7\}$, 
$\cA_1=\cA\setminus\{H_{j_1}\}$, 
$\cA_2=\cA_1\setminus\{H_{j_2}\}$ and 
$\cA_3=\cA_2\setminus\{H_{\infty}\}$. 
Observe that 
$n_{\cA,H_{j_1}}=n_{\cA_1,H_{j_2}}= 
n_{\cA_2,H_{\infty}}=6$ and 
$\cA_3$ is super solvable. 
Therefore we conclude that $\cA\in\cI_{13}$ 
with $\exp(\cA)=(1,5,7)$. 
\end{subsection} 
\begin{subsection}{Freeness of $\cA$} 
We show (3) of Proposition \ref{13Descrip}. 
 
Since $L(\cA)=L_0$, it follows that 
$F(\cA)=[21,3,3,3]$, $n_{\cA,H}=6$ for any $H\in\cA$, 
and $\chi({\cA},t)=(t-1)(t-6)^2$. 
Therefore, we have only to show that $\cA\in\cF_{13}$. 
We show it in terms of Yoshinaga's criterion in \cite{Y}. 
 
Let 
$(\cA'',m)$ be 
the Ziegler restriction of $\cA$ onto $H_\infty$. 
It is defined by 
$$ 
y^3 (-\sqrt{3}x-y)^3(\sqrt{3}x-y)^3 \{ 
(\lambda+1)y+\sqrt{3}(1-\lambda)x)\} 
\{ 
\sqrt{3}\lambda x+(\lambda-2)y\} 
\{ 
(1-2\lambda)y-\sqrt{3}x\}=0. 
$$ 
By the change of coordinates 
$$ 
u=y+\sqrt{3}x,\ v=y-\sqrt{3}x, 
$$ 
the defining equation of $(\cA'',m)$ becomes 
$$ 
u^3v^3(u+v)^3 (u+\lambda v)(u+v-\lambda u)(\lambda u+\lambda v-v)=0. 
$$ 
Now recall the following. 
 
\begin{prop}[\cite{Y}]\label{yoshinaga} 
Let $\cB$ be a central arrangement in $\bC^3$, $H \in \cB$ and 
let $(\cB'',m)$ be the Ziegler restriction of $\cB$ onto $H$. Assume that 
$\chi(\cB,t)=(t-1)(t-a)(t-b)$, and 
$\exp(\cB'',m)=(d_1,d_2)$. 
Then $\cB$ is free if and only if 
$ab=d_1d_2$. 
\end{prop} 
 
Also, recall that, 
for a central multiarrangement $(\cC,m)$ in $\bC^2$ 
with $\exp(\cC,m)=(e_1,e_2)$ $(e_1 \leq e_2)$, 
it holds that $e_1+e_2=|m|=\sum_{H \in \cC} m(H)$ and 
$e_1=\min_{d \in \bZ} \{d \mid D(\cC,m)_d \neq 0\}$. 
This follows from the fact that 
$D(\cC,m)$ is a rank two free module. For example, see \cite{A}. 
 
Now since $\chi(\cA,t)=(t-1)(t-6)^2$, 
it suffices to show that 
every homogeneous derivation of degree five is zero. 
 
Assume that $\theta \in D(\cA'',m)$ is homogeneous of 
degree five and show that $\theta=0$. 
To check it, first, let us introduce a submultiarrangement $(\cB,m')$ 
of $(\cA'',m)$ defined by 
$$ 
u^3v^3 (u+v)^3=0. 
$$ 
The freeness of $(\cB,m')$ is well-known. 
In fact, we can give its explicit basis as follows. 
$$ 
\partial_1=(u+2v)u^3\partial_u- 
(2u+v)v^3\partial_v,\ 
\partial_2=(u+3v)vu^3 \partial_u+ 
(3u+v)uv^3 \partial_v. 
$$ 
So $\exp(\cB,m')=(4,5)$. Since $D(\cB,m') \supset D(\cA'',m)$, 
there are scalars $a,b,c \in \bC$ such that 
$$ 
\theta=(au+bv)\partial_1+c\partial_2. 
$$ 
The scalars $a,b$ and $c$ are determined by the tangency conditions 
to the three lines 
$u+\lambda v=0,\ u+v-\lambda u=0,\ \lambda u+\lambda v-v=0$. 
Then a direct computation shows that $a,b$ and $c$ satisfy 
$$ 
\lambda(\lambda-1) 
\begin{pmatrix} 
\lambda & -1 &-\lambda(\lambda+1)\\ 
1 & \lambda-1 & -(\lambda-1)(\lambda-2)\\ 
\lambda-1 & -\lambda & -\lambda(\lambda-1)(2\lambda-1) 
\end{pmatrix} 
\begin{pmatrix} 
(\lambda^2-\lambda+1)a\\ 
(\lambda^2-\lambda+1)b\\ 
c 
\end{pmatrix}=0. 
$$ 
The above linear equations imply that 
$$ 
-\lambda(\lambda-1) 
(\lambda-2)(2\lambda-1)(\lambda+1)(\lambda^2-\lambda+1) 
\begin{pmatrix} 
(\lambda^2-\lambda+1)a\\ 
(\lambda^2-\lambda+1)b\\ 
c 
\end{pmatrix}=0. 
$$ 
Since $\lambda\in\bC\setminus Z$, it holds that $a=b=c=0$. 
Hence 
$\theta=0$, that is to say, $D(\cA'',m)_5=0$. 
Now apply 
Proposition \ref{yoshinaga} to 
show that $\cA$ is free with $\exp(\cA)=(1,6,6)$. 
 
\medskip 
 
We can also construct the basis 
of $D(\cA)$ explicitly and give 
an alternative proof of the freeness of $\cA$ 
by Theorem \ref{saito}. 
However, we omit to describe it here 
because of its lengthy. 
\end{subsection} 
\begin{subsection}{Moduli of the lattice $L_0$}\label{13moduli} 
We show (4) and (5) of Proposition \ref{13Descrip}. 
 
Since three parallel affine lines of $\cA$ are transferred 
each other by the rotation, the ratio of the distances among them 
coincides up to its ordering. 
Looking at the equations $h_2$, $h_6$, $h_7$, 
we calculate the ratio as follows. 
$$ 
\left\{(\lambda+1)-(\lambda-2)\right\}: 
\left\{(\lambda-2)-(-2\lambda+1)\right\}: 
\left\{(-2\lambda+1)-(\lambda+1)\right\} 
=(-1):(1-\lambda):\lambda. 
$$ 
 
We assume that 
$\cA_{\lambda}$ is transferred to $\cA_{\lambda'}$ 
by a $\PGL(3,\bC)$-action $\rho$ and 
determine the relation between $\lambda$ and $\lambda'$. 
Note that $H_{13}$ is distinguished among $\cA$ by the 
feature $F_{H_{13}}(\cA)=[3,0,3]$. Thus we may assume 
that $\rho$ preserves $H_{13}=H_{\infty}$ and hence 
the ratio above up to the ordering. Therefore we have 
$\lambda'\in S_{\lambda}$ where 
$$S_{\lambda}=\left\{ 
\lambda,\lambda^{-1},1-\lambda, 
(1-\lambda)^{-1}, 1-\lambda^{-1}, 
(1-\lambda^{-1})^{-1} 
\right\}.$$ 
In fact, any $\lambda'\in S_{\lambda}$ gives $\cA_{\lambda'}$ 
equivalent to $\cA_\lambda$. 
Actually, $\cA_{\lambda^{-1}}$ is realized by 
$\rho_1=[(x,y)\mapsto(-\lambda x,\lambda y)]$, which sends $H_i$ to 
$H_{\sigma_1(i)}$ where $\sigma_1=(1,3)(4,9)(5,8)(6,7)(11,12)$, while 
$\cA_{1-\lambda}$ is realized by $\rho_2=[(x,y)\mapsto(x,-y)]$, 
which sends $H_i$ to $H_{\sigma_2(i)}$ where 
$\sigma_2=(1,4)(2,6)(3,5)(8,9)(10,11)$. 
Composing them, we also have the realizations for 
$\lambda'\in\{(1-\lambda)^{-1}, 1-\lambda^{-1}, (1-\lambda^{-1})^{-1}\}$. 
 
Now it is shown that the set of $\cA_\lambda$ parametrized by 
$\left(\bC\setminus Z\right)/\sim$, 
where $\lambda\sim\lambda'$ holds if and only if 
$\lambda'\in S_{\lambda}$, 
forms a non-trivial one dimensional family of 
the realizations of the lattice $L_0$. 
We show that, it gives the whole moduli space of 
the realizations of $L_0$ 
by determining all the realizations up to the 
$\PGL(3,\bC)$-action. 
 
We may set $H_{13}$ as the infinity line $H_{\infty}$, 
$h_1=x$, 
$h_2=y-1$, 
$h_5=x-1$, 
$h_6=y$, 
and 
$$ 
\{1,4,7,10,11\}=(0,p), 
\quad 
\{3,6,9,10,12\}=(q,0). 
$$ 
Then we have 
$\{1,6,8\}=(0,0)$, 
$\{2,5,8,11,12\}=(1,1)$ 
and 
$$ 
h_8=x-y,\ 
h_{10}=px+qy-pq,\ 
h_{11}=(p-1)x+y-p,\ 
h_{12}=x+(q-1)y-q. 
$$ 
Since 
$(q,0)\in H_3\parallel H_8$, 
$(0,p)\in H_4\parallel H_8$, 
$(0,p)\in H_7\parallel H_6$ and 
$(q,0)\in H_9\parallel H_1$, 
we have 
$$ 
h_3=x-y-q,\ 
h_4=x-y+p,\ 
h_7=y-p,\ 
h_9=x-q. 
$$ 
Since $\{2,4,9\}=(q,1)\in H_4$ and 
$\{3,5,7\}=(1,p)\in H_3$, we have 
$p+q=1$.  Thus all the realizations are 
expressed in terms of at most one parameter. 
 
Now set $p=1-\lambda$, $q=\lambda$ and 
apply the following $\PGL(3,\bC)$-action 
$$\left[ 
x\mapsto -( \sqrt{3}x+y -\lambda-1 )/3,\ 
y\mapsto -\left( 2y+\lambda-2 \right)/3 
\right].$$ 
Then we can recover the equations 
$\{h_i\mid 1\leq i\leq 12\}$ 
in \S \ref{13def}. 
Therefore it is verified that 
$\Vc_\bC(L_0)=\left(\bC\setminus Z\right)/\sim$. 
Now it is easy to see that the symmetric group of $L_0$ is 
given by the rotations and the equivalence relations. 
The former corresponds to $\bZ/3\bZ$, and the latter 
corresponds to $\fS_3$. 
\end{subsection} 
\begin{subsection}{Non-recursive freeness of $\cA$} 
We show (6) of Proposition \ref{13Descrip}. 
\begin{step} 
We show that $\cA\in\cR_{13}$ if $\lambda\in W$. 
\end{step} 
Note that $W$ is the union of the equivalent classes of 
$\sqrt{-1}$ and $(1+\sqrt{5})/2$ with respect to 
the equivalence relation $\sim$ in \S \ref{13moduli}. 
Thus we have only to consider the cases when 
$\lambda=\sqrt{-1}$ or $(1+\sqrt{5})/2$. 
 
In the case when $\lambda=\sqrt{-1}$, we 
set $H_{14}=\overline{\{1,2\}\{5,10\}}$ and 
$\cA_1=\cA\cup\{H_{14}\}$.  Then, 
the intersection points on $H_{14}$ is calculated as follows. 
$$ 
\{3,4,8,13,14\}, \{1,2,14\}, \{5,10,14\}, \{7,12,14\}, 
\{6,14\}, \{9,14\}, \{11,14\}. 
$$ 
We set 
$\cA_2=\cA_1\setminus\{H_{6}\}$, 
$\cA_3=\cA_2\setminus\{H_{9}\}$, 
$\cA_4=\cA_3\setminus\{H_{1}\}$ and 
$\cA_5=\cA_4\setminus\{H_{13}\}$. 
Then it is easy to see that 
$n_{\cA_1,H_{6}}=n_{\cA_2,H_{9}}=7$ and 
$n_{\cA_3,H_{1}}=n_{\cA_4,H_{13}}=6$. 
Now, since $\exp(\cA)=(1,6,6)$, 
we have $\cA_i\in\cF_{15-i}$ for $1\leq i\leq 5$ 
by Theorem \ref{ad}. 
Since $\cA_5\in\cF_{10}=\cI_{10}$ by 
Proposition \ref{9-10}, it follows that $\cA\in\cR_{13}$. 
 
In the case when $\lambda=(1+\sqrt{5})/2$, we set 
$H_{14}=\overline{\{1,4\}\{2,3\}}$ and 
$\cA_1=\cA\cup\{H_{14}\}$.  Then, 
the intersection points on $H_{14}$ is calculated as follows. 
$$ 
\{1,4,7,10,11,14\}, \{2,3,14\}, \{5,6,14\}, \{8,14\}, 
\{9,14\}, \{12,14\}, \{13,14\}. 
$$ 
We set 
$\cA_2=\cA_1\setminus\{H_{9}\}$, 
$\cA_3=\cA_2\setminus\{H_{12}\}$, 
$\cA_4=\cA_3\setminus\{H_{13}\}$ and 
$\cA_5=\cA_4\setminus\{H_{2}\}$. 
Then it is easy to see that 
$n_{\cA_1,H_{9}}=n_{\cA_2,H_{12}}=7$ and 
$n_{\cA_3,H_{13}}=n_{\cA_4,H_{2}}=6$. 
Now, since $\exp(\cA)=(1,6,6)$, 
we have $\cA_i\in\cF_{15-i}$ for $1\leq i\leq 5$ 
by Theorem \ref{ad}. 
Since $\cA_5\in\cF_{10}=\cI_{10}$ by 
Proposition \ref{9-10}, it follows that $\cA\in\cR_{13}$. 
 
Therefore we conclude that $\cA\in\cR_{13}$ if 
$\lambda\in W$. 
 
\medskip 
 
Next, we 
assume that $\cA\in\cR_{13}$ 
for some 
$\lambda\in\bC\setminus\left(Z\cup W\right)$ 
and deduce the contradiction. 
Recall that $\cA\in\cS_{13}$ with $\exp(\cA)=(1,6,6)$. 
Since $\cA\in\cR_{13}$, there exists a line 
$L\subset\bP_{\bC}^2$ such that 
$\cA_1=\cA\cup\{L\}\in\cF_{14}$. 
Note that $n_{\cA_1,L}=7$ by Theorem \ref{ad}. 
 
\begin{step}\label{FLcand} 
We show that 
$F_L(\cA_1)\in\left\{[4,2,0,0,1], [3,3,0,1], [3,2,2], [2,4,1], [1,6]\right\}$. 
\end{step} 
Let $P,Q\in\bigcup_{i=2}^4M_i(\cA)$ with $P\neq Q$. 
Then $\overline{PQ}\in\cA$ unless $P,Q\in M_2(\cA)$. 
Also note that $3$ points of $M_2(\cA)$ are not collinear 
since $\lambda\neq-\omega^{\pm1}$. 
Therefore we have $\sum_{i=3}^5F_{L,i}(\cA_1)\leq2$, 
where the equality holds only when $F_{L,3}(\cA_1)=2$. 
Since $n_{\cA_1,L}=7$ and $\mu_{\cA_1,L}=13$, we have 
only $5$ possibilities of $F_L(\cA_1)$ appearing 
in the assertion. 
 
In the rest of this subsection, we show that each possibilities 
of $F_L(\cA_1)$ above cannot occur. For the later use, 
we present the list of the elements of $M_1(\cA)$ explicitly as follows. 
$$M_1(\cA)=\left\{\begin{array}{lllllll} 
 \{1,2\}, &\{2,10\}, &\{4,5\}, &\{5,10\}, &\{7,8\}, &\{8,10\}, &\{10,13\}, 
\\ 
 \{2,3\}, &\{3,11\}, &\{5,6\}, &\{6,11\}, &\{8,9\}, &\{9,11\}, &\{11,13\}, 
\\ 
 \{3,1\}, &\{1,12\}, &\{6,4\}, &\{4,12\}, &\{9,7\}, &\{7,12\}, &\{12,13\} 
\end{array}\right\} \leqno{(\ast)}$$ 
\begin{step} 
We analyze the case when $F_L(\cA_1)=[4,2,0,0,1]$. 
\end{step} 
Since $|L\cap M_4(\cA)|=F_{L,5}(\cA_1)=1$, 
we may assume $P=\{1,4,7,10,11\}\in L$ by symmetry. 
Since $L\not\in\cA$, we have 
$L\cap M_1(\cA)\subset\{\{2,3\}, \{5,6\}, \{8,9\}, \{12,13\}\}$. 
Note that $\sigma_2$ in \S \ref{13moduli} 
permutes $\{2,3\}$ and $\{5,6\}$, 
while the composition of the rotation and $\sigma_1$ in 
\S \ref{13moduli} permutes $\{5,6\}$ and $\{8,9\}$. 
%, both preserving $\{1,4,7,10,11\}$. 
Since $|L\cap M_1(\cA)|=F_{L,2}(\cA_1)=2$, we may assume 
that $L=\overline{PQ}$ where $Q=\{2,3\}$ and that 
$\{5,6\}\in L$ or $\{12,13\}\in L$. 
By calculation, we see 
that $\{5,6\}\in L$ implies $\lambda=(1\pm\sqrt{5})/2$ 
and 
that $\{12,13\}\in L$ implies $\lambda=1\pm\sqrt{-1}$. 
Since $\lambda\not\in W$, both cases cannot occur. 
Therefore we conclude that $F_L(\cA_1)\neq[4,2,0,0,1]$. 
 
Before looking into other cases, we give a useful observation. 
\begin{step}\label{obs3} 
We show that if 
$P_i\in H_{9+i}\cap M_1(\cA)$ for $1\leq i\leq 3$, 
then $\{P_i\mid 1\leq i\leq 3\}\not\subset L$. 
\end{step} 
Assume that $\{P_i\mid 1\leq i\leq 3\}\subset L$. 
Recall that $L\not\in\cA$. 
By the actions of the rotation and $\sigma_1,\sigma_2$ 
in \S \ref{13moduli},  we may assume 
$P_1=\{2,10\}$, $P_2=\{6,11\}$ and 
$P_3\in\{\{7,12\},\{12,13\}\}$. 
By calculation, we see that 
$\{7,12\}\in\overline{P_1P_2}$ or 
$\{12,13\}\in\overline{P_1P_2}$ 
imply $\lambda\in\{1/2,-\omega^{\pm1}\}\subset W$, 
a contradiction. 
Thus we conclude that $\{P_i\mid1\leq i\leq 3\}\not\subset L$. 
\begin{step} 
We analyze the case when $F_L(\cA_1)=[3,3,0,1]$. 
\end{step} 
Since $F_{L,4}(\cA_1)=1$, we may assume 
$P=\{2,6,7,13\}\in L$ by symmetry. It follows that 
$$ 
L\cap M_1(\cA)\subset\left\{ 
\{1,3\},\{4,5\},\{8,9\}, 
\{1,12\},\{3,11\}, 
\{4,12\},\{5,10\}, 
\{8,10\},\{9,11\} 
\right\}. 
$$ 
Considering the actions of $\sigma_1$, $\sigma_2$ in \S \ref{13moduli} 
and using Step \ref{obs3} and $F_{L,2}(\cA_1)=3$, we may assume that 
$\left\{\{1,3\},\{4,5\}\right\}\subset L\cap M_1(\cA)$ or 
$\left\{\{1,3\},\{5,10\}\right\}\subset L\cap M_1(\cA)$. 
Since $H_2$ and $\overline{\{1,3\}\{4,5\}}$ are parallel to 
$x$-axis and $y$-axis respectively, the former cannot occur. 
Thus we have $L=\overline{\{1,3\}\{5,10\}}$. 
Since $P\in L$, by calculation, we have 
$\lambda=1/2\in W$, a contradiction. 
Therefore we conclude that $F_L(\cA_1)\neq[3,3,0,1]$. 
\begin{step} 
We analyze the case when $F_L(\cA_1)=[3,2,2]$. 
\end{step} 
Since $F_{L,3}(\cA_1)=2$, we may assume that 
$L\cap M_2(\cA)=\{\{1,6,8\}\{2,4,9\}\}$ 
by the symmetry of the rotation.  Therefore 
we have $L=\overline{\{1,6,8\}\{2,4,9\}}$ and 
$$ 
L\cap M_1(\cA)\subset\left\{ 
\{3,11\}, 
\{5,10\}, 
\{7,12\}, 
\{10,13\}, 
\{11,13\}, 
\{12,13\} 
\right\}. 
$$ 
Note that 
$\sigma_2$ in \S \ref{13moduli} permutes 
$\{3,11\}$ and $\{5,10\}$, 
while the composition of the rotation and 
$\sigma_1$ in \S \ref{13moduli} permutes 
$\{5,10\}$ and $\{7,12\}$. 
Since $F_{L,2}(\cA_1)=2$, we may assume that 
$\{3,11\}\in L\cap M_1(\cA)$. 
By calculation, it follows that 
$\lambda=\pm\sqrt{-1}\in W$, a contradiction. 
Therefore we conclude that 
$F_L(\cA_1)\neq[3,2,2]$. 
\begin{step} 
We analyze the case when $F_L(\cA_1)=[2,4,1]$. 
\end{step} 
Since $F_{L,3}(\cA_1)=1$, we may assume that 
$P=\{\{1,6,8\}\}\in L\cap M_2(\cA)$ 
by the symmetry of the rotation.  Therefore we have 
$$L\cap M_1(\cA)\subset\left\{\{2,3\}, \{4,5\}, \{9,7\}\right\} 
\cup H_{10}\cup H_{11}\cup H_{12}.$$ 
By Step \ref{obs3} and $F_{L,2}(\cA_1)=4$, we see that 
$\left|L\cap\left\{\{2,3\}, \{4,5\}, \{9,7\}\right\}\right|\geq2$. 
Note that 
the composition of the rotation and $\sigma_2$ in \S \ref{13moduli} 
permutes $\{2,3\}$ and $\{4,5\}$, 
while the composition of the rotation and 
$\sigma_1$ in \S \ref{13moduli} 
permutes $\{4,5\}$ and $\{9,7\}$. 
Thus we may assume 
$L=\overline{\{2,3\}\{4,5\}}$. 
By calculation, we see that 
$P\in L$ implies $\lambda=(1\pm\sqrt{5})/2\in W$, 
a contradiction. 
Therefore we conclude that 
$F_L(\cA_1)\neq[2,4,1]$. 
\begin{step} 
We analyze the case when $F_L(\cA_1)=[1,6]$. 
\end{step} 
Let $U_i$ be the set of points 
appearing in the $i$-th column of 
the list $(\ast)$ in Step \ref{FLcand} 
for $1\leq i\leq 7$. 
By Step \ref{obs3}, we have 
$\left|L\cap\bigcup_{i=2,4,6,7}U_i\right|\leq2$. 
Since $L\not\in\cA$, we have 
$|L\cap U_i|\leq1$ for $i=1,3,5$. 
Thus it follows that 
$6=F_{L,2}(\cA_1)=\left|L\cap M_1(\cA)\right|\leq 2+3=5$, 
which is a contradiction. 
Therefore we conclude that 
$F_L(\cA_1)\neq[1,6]$. 
 
\medskip 
 
Now the proof of $\cA\not\in\cR_{13}$ for 
$\lambda\in\bC\setminus\left(Z \cup W\right)$ 
is completed. 
\end{subsection} 
\end{section} 
\begin{section}{An arrangement with $15$ hyperplanes}\label{15} 
In this section we present a further example of a free but not 
recursively free arrangement whose intersection lattice defines 
a one dimensional moduli space. One obtains this arrangement for 
example as a deformation of a simplicial arrangement or as a 
subarrangement of a restriction of a reflection arrangement. 
But it also showed up in several other experiments. 
 
In contrast to all other results of this paper, we prove all the 
properties of this intersection lattice with the computer. 
The techniques presented here may also be used to verify all the 
claims of Section \ref{13}. 
For instance, the following algorithm may be used to prove that 
all arrangements in a given moduli space are free. 
 
\algo{ModuliSpaceIsFree}{$L$}{Checks that all arrangements in $\Vc(L)$ are free} 
{The intersection lattice $L=L(\cA)$ of a free arrangement $\cA$.} 
{``True'' if all realizations of $L$ are free, ``Terao's open problem 
is false'' else.} 
{ 
\item Realize the moduli space as a locally closed set 
(use the algorithm in \cite{p-C10b}): Construct ideals $I_0,J$ such that 
the moduli space is $\Vc(I_0)\setminus \Vc(J)$, preferably $\Vc(I_0)$ 
should be as small as possible. 
\item Compute a primary decomposition of $I_0$. 
\item If $I_0$ has dimension zero and is a primary ideal, then return 
``True'' (the lattice is rigid). 
\item Perform the next steps for each prime $I$ in the decomposition of $I_0$. 
\item We now have a realization $\cA_K$ of the intersection lattice over 
the field of fractions $K$ of the coordinate ring of $I$ (this is 
an integral domain since $I$ is prime). 
\item If $\cA_K$ is not free in $K^r$, then return ``Terao's 
open problem is false''. 
\item If $\cA_K$ is free in $K^r$, then compute a basis $B$ of 
its logarithmic derivation module. 
\item Compute the Saito determinant $d$ of $B$; this is an element 
of $K$, i.e., a rational function modulo $I$ in the variables defining $I$. 
\item Include the information that $d/Q(\cA_K)$ is well defined and 
non-zero into the ideal $I$, and iterate the above until $I$ is zero dimensional. 
\item The zero set of $I$ is now finite. Exclude the solutions which 
give the wrong intersection lattice (they occur because we have ignored the ideal $J$). 
\item If all remaining solutions are free, return ``True'', 
otherwise return ``Terao's open problem is false''. 
} 
 
Let us now define the desired intersection lattice with 15 planes. 
Define $R\subseteq \bZ[t]^3$ as 
\begin{eqnarray*} 
R &=& \{ (1,0,0),(1,1,0),(1,0,1),(1,1,1),(1,t,1),(0,1,0),(2,1,1),\\ 
&& (t+1,t,1),(t+1,1,1),(2t,t,1),(1,-t+1,1),\\ 
&& (-3t+1,t^2-3t+1,-t),(3t-1,t,t),\\ && (-3t+1,-t^2,-t),(3t-1,2t-1,t) \}. 
\end{eqnarray*} 
We denote $\cA_t = \{\ker_{\bQ(t)} \alpha\mid \alpha \in R\}$ the arrangement defined by $R$ in the space $\bQ(t)^3$. 
For any field extension $K/\bQ$ and each $\lambda \in K$, we also have a corresponding arrangement 
\[ \cA_\lambda:=\{ \ker_K \alpha(\lambda) \mid \alpha\in R\} \] 
in $K^3$. 
\begin{figure}[hbt] 
\begin{center} 
\setlength{\unitlength}{0.75pt} 
\begin{picture}(400,400)(100,200) 
\moveto(496.773982019981493284007282848,364.222912360003364857453221300) 
\lineto(103.226017980018506715992717152,364.222912360003364857453221300) 
\moveto(499.799899899874824737086828568,408.944271909999158785636694675) 
\lineto(100.200100100125175262913171431,408.944271909999158785636694675) 
\moveto(495.715111622386654550161466794,441.177604139103725326976197484) 
\lineto(104.284888377613345449838533206,441.177604139103725326976197484) 
\moveto(462.003080212687308662930982098,517.281720662691514734747508724) 
\lineto(182.718279337308485265252491277,237.996919787312691337069017902) 
\moveto(441.421356237309504880168872421,541.421356237309504880168872421) 
\lineto(158.578643762690495119831127579,258.578643762690495119831127579) 
\moveto(406.952750824954505690978417484,569.000322754056281424622386528) 
\lineto(130.999677245943718575377613472,293.047249175045494309021582516) 
\moveto(308.944271909999158785636694675,599.799899899874824737086828568) 
\lineto(308.944271909999158785636694675,200.200100100125175262913171431) 
\moveto(367.958188185589464486680994098,211.899801545739966944654699262) 
\lineto(223.491767387609719016567995648,584.787689910687449561575785165) 
\moveto(444.811903076327848709446111418,262.052500104524442500854566083) 
\lineto(173.745865467501343672573169961,555.112518880487948495621622922) 
\moveto(499.968602029388036590795160647,396.456244589107931398792724109) 
\lineto(100.031397970611963409204839353,396.456244589107931398792724109) 
\moveto(432.492660930138327285243171407,549.818873309244309090703667076) 
\lineto(150.181126690755690909296332923,267.507339069861672714756828593) 
\moveto(422.256045515541062845432718669,241.717154009343286710386751204) 
\lineto(151.716113756179965018531360123,534.208379322715313821763695360) 
\moveto(482.566108371663406313081719590,481.667717453263334783894740716) 
\lineto(113.729449702885217465211784879,327.176363095420508928667332648) 
\moveto(279.130032210001949593332228440,598.908130664496578018818014653) 
\lineto(279.130032210001949593332228440,201.091869335503421981181985346) 
\thinlines 
\strokepath 
\put(470,560){\Large $\infty$} 
\end{picture} 
\end{center} 
\caption{An arrangement $\cA_\lambda$ with $15$ hyperplanes.} 
\end{figure}
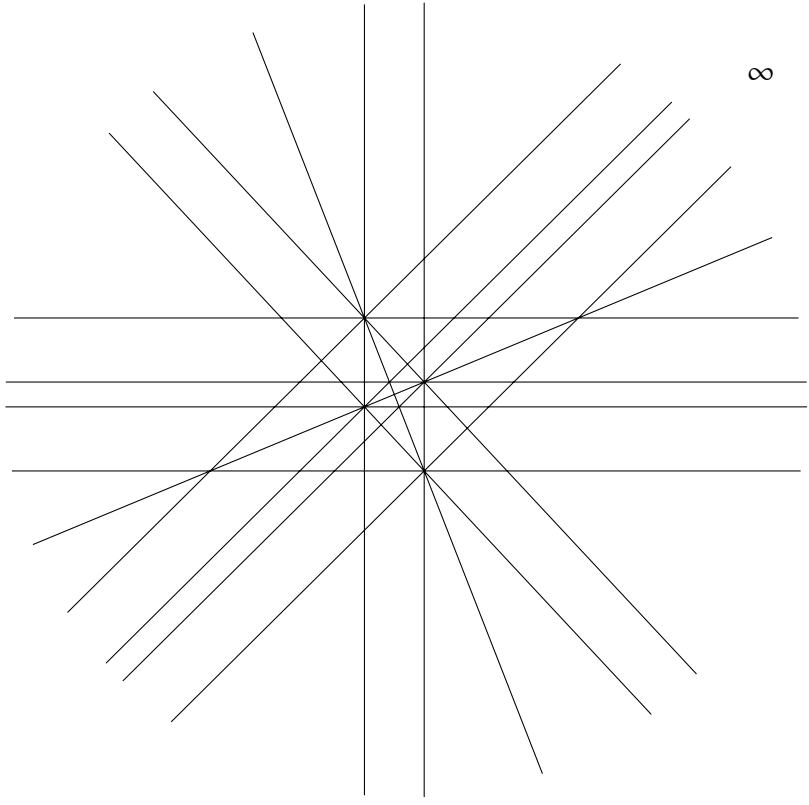 
\begin{prop} 
The intersection lattice $L(\cA_t)$ has the following properties. 
\item[(1)] 
The moduli space of $L(\cA_t)$ is one dimensional. More precisely, 
$$\Vc_\bC(L(\cA_t)) = \{ \cA_\lambda \mid \lambda \in \bC\setminus Z\} 
\quad\text{where}\quad 
Z=\left\{0, 1, \frac{1}{2},\frac{3}{2}\pm \sqrt{2},\frac{-1\pm \sqrt{5}}{2}\right\}.$$ 
\item[(2)] 
If $\lambda\in \{0, 1, 1/2\}$ then $|\cA_t|\neq|\cA_\lambda|$. 
\item[(3)] 
If 
$\lambda\in \left\{(3\pm2\sqrt{2})/2,(-1\pm \sqrt{5})/2\right\}$ 
then $|\cA_t|=|\cA_\lambda|$. However, $L(\cA_\lambda)$ is the intersection lattice of the simplicial arrangement $\cA(15,1)$ in Gr\"unbaum's notation (see \cite{p-G-09}). 
Thus in this case, the arrangements $\cA_\lambda$ are free with exponents $\lmultiset 1,5,9\rmultiset$. 
\item[(4)] 
The arrangement $\cA_t$ is free but not recursively free with exponents $\lmultiset 1,7,7\rmultiset$. 
\item[(5)] 
The arrangements $\cA_\lambda$ for $\lambda\in \{-1, 1/3, 2, (1\pm\sqrt{-1})/2, 2/3\}$ are recursively free. 
All other arrangements $\cA_\lambda$ with $\lambda\notin Z$ are free but not recursively free. 
\item[(6)] 
The symmetry group of $L(\cA_t)$ is a wreath product, $\Aut(L(\cA_t)) \cong \bZ/2\bZ \wr \fS_3$, thus it is isomorphic to the reflection group of type $B_3$. 
\end{prop} 
\begin{proof} 
(1) We compute the moduli space using the algorithm of \cite{p-C10b}. Notice that contrary to the example with $13$ planes, there are no automorphisms of the intersection lattice which induce projectivities here.\\ 
(4) We check the freeness results with Algorithm \ref{ModuliSpaceIsFree}.\\ 
(5) If $\tilde\cA$ is an arrangement consisting of $\cA_t$ plus an extra hyperplane $H$, then there are three possible cases. 
The hyperplane $H$ does not contain any intersection point of $\cA_t$, $H$ goes through exactly one ``old'' intersection point, or $H$ goes through at least two ``old'' intersection points. The first two cases are easy to handle. Thus we have finitely many possible hyperplanes going through at least two points of $\cA_t$ to check. For each such hyperplane $H$ we perform the following steps. 
\begin{itemize} 
\setlength{\itemindent}{-0.4\leftmargin} 
\item Compute the new intersection points $p_1,\ldots,p_k$ between $H$ and lines in $\cA_t$. 
\item Consider all solutions $t$ for which some $p_i$ coincides with an intersection point of $\cA_t$. 
\item For all these (finitely many) specializations $\lambda$ of $t$, check the recursive freeness of $\cA_\lambda$. 
\end{itemize} 
Specializing $t$ to $\lambda=-1, 1/3, 2, (1+\sqrt{-1})/2, (1-\sqrt{-1})/2, 2/3$ and including the kernels of $(2,-1,-1)$, $(0,0,1)$, $(2,0,1)$, $(2,0,1-\sqrt{-1})$, $(2,0,1+\sqrt{-1})$, $(4,3,3)$ respectively into $\cA_\lambda$ yields inductively free arrangements with exponents $\lmultiset 1,7,8\rmultiset$.\\ 
(2), (3), and (6) are direct computations performed with {\sc Singular} and {\sc GAP}. 
\end{proof} 
 
\end{section} 
 
\end{document}